\documentclass[11pt]{article}

\usepackage{amsmath}
\usepackage{amssymb}
\usepackage{amsthm}
\usepackage{latexsym}
\usepackage{color}
\usepackage{graphicx}
\usepackage{color}
\usepackage{caption}
\usepackage{subcaption}
\usepackage[normalem]{ulem}

\usepackage{xcolor}

\DeclareSymbolFont{calletters}{OMS}{cmsy}{m}{n}
\DeclareSymbolFontAlphabet{\mathcal}{calletters}

\def\be{\begin{eqnarray}}
\def\ee{\end{eqnarray}}

\def\b*{\begin{eqnarray*}}
\def\e*{\end{eqnarray*}}

\newtheorem{Theorem}{Theorem}[section]
\newtheorem{Definition}[Theorem]{Definition}
\newtheorem{Proposition}[Theorem]{Proposition}

\newtheorem{Assumption}[Theorem]{Assumption}

\newtheorem{Corollary}[Theorem]{Corollary}
\newtheorem{Remark}[Theorem]{Remark}
\newtheorem{Example}[Theorem]{Example}

\makeatletter \@addtoreset{equation}{section}

\usepackage{color}

\newcommand{\bea}{\begin{eqnarray}}
\newcommand{\bes}{\begin{subequations}}
\newcommand{\ees}{\end{subequations}}
\newcommand{\bgt}{\begin{gather}}
\newcommand{\egt}{\begin{gather}}
\newcommand{\eea}{\end{eqnarray}}
\newcommand{\beaa}{\begin{eqnarray*}}
\newcommand{\eeaa}{\end{eqnarray*}}

\def \E{\mathbb{E}}
\def \F{\mathbb{F}}

\def \P{\mathbb{P}}
\def \Q{\mathbb{Q}}
\def \R{\mathbb{R}}

\def\Ac{{\cal A}}

\def\Fc{{\cal F}}

\def\Hc{{\cal H}}

\def\Pc{{\cal P}}

\def\Cb{\overline{C}}

\def \Om{\Omega}
\def \om{\omega}

\def \eps{\varepsilon}

\def \0{\mathbf{0}}

\def \x{\times}

\def \T{\mathbb{T}}

\def\1{\mathbf{1}}
\def\xr{{\rm x}}

\def\CT{C([0,T])}

\def\DT{D([0,T])}

 \def\Ninfty#1{\|#1\|}    
 \def\Cb{{\mathbb C}}

  \def\vs#1{\vspace{2mm}}

\title{A  ${\mathbb C}^{0,1}$-functional It\^{o}'s formula  and  its  applications in mathematical finance}

\author{
Bruno Bouchard
\footnote{CEREMADE, Universit\'e Paris-Dauphine, PSL, CNRS.  bouchard@ceremade.dauphine.fr. } 
\and 
Gr\'egoire Loeper
\footnote{Monash University, School of Mathematical Sciences \& Centre for Quantitative Finance and Investment Strategies (CQFIS),  gregoire.loeper@monash.edu. CQFIS has been supported by BNP Paribas.}
\and 
Xiaolu Tan
\footnote{Department of Mathematics, The Chinese University of Hong Kong. xiaolu.tan@cuhk.edu.hk.}}

\begin{document}
\maketitle

\begin{abstract} 
	Using Dupire's notion of vertical derivative, we provide a functional (path-dependent) extension of the It\^{o}'s formula of   Gozzi and Russo (2006) that applies to ${C}^{0,1}$-functions of continuous weak Dirichlet processes.
	It is motivated and illustrated by its applications to the hedging or superhedging problems  of path-dependent options in mathematical finance, in particular in the case of model uncertainty.  
\end{abstract}

\section{Introduction}

	Let $X$ be a $\R^d$--valued continuous semimartingale with (unique) decomposition $X = X_0 + M +A$,
	where $M$ is a continuous martingale and $A$ is a finite variation process such that $M_0 = A_0 = 0$.
	Let $f : [0,T]\x \R^d \longrightarrow \R$ be a $C^{1,2}$-function,
	then, It\^{o}'s Lemma says that 
	\begin{eqnarray}\label{eq: intro ito}
		f(t, X_t)
		&=&
		f(0, X_{0})
		+
		\int_{0}^{t} \nabla_{x} f(s, X_{s}) dM_s + \Gamma^{f}_t, ~\mbox{a.s.},
	\end{eqnarray}
	 in which $\nabla_{x} f$ is the gradient in space of $f$, viewed as a line vector, and $\Gamma^{f}$ is a continuous process with finite variation, given by 
	$$
		\Gamma^f_t 
		:=
		\int_{0}^{t} \partial_{t} f(s,X_s) ds 
		+
		\sum_{1 \le i \le d} \int_0^t \nabla_{x^i} f(s, X_s) dA_s
		+
		\frac12  {\sum_{1\le i, j\le d}} \int_{0}^{t} \nabla^2_{x^{ {i}}x^{ {j}}} f(s,X_s) d [X^{ {i}},X^{ {j}} ]_{s}.
	$$
	If we assume in addition that $X$ and $f(\cdot,X)$ are both  local martingales, then $\Gamma^f \equiv 0$, a.s., 
	so that the formula does not involve the partial derivatives $\partial_t f$ and $(\nabla^2_{x^{i}x^{j}}f )_{i,j\le d}$ any more.
	In this case, one might expect that the above formula still holds even if $f $ is only ${C}^{0,1}$. 
	This was in fact achieved by the stochastic calculus  via regularization theory that was developed in \cite{russo1993forward, russo1995generalized, russo1996ito, gozzi2006weak, bandini2017weak, coquet2006natural}.  	\vspace{0.5em}
	
	In this theory,  notions of orthogonal (or zero energy)  and weak Dirichlet processes have been introduced (see below for a precise definition),  
	which generalize respectively the notions of finite variation processes and of  semimartingales.
	It is proved that, for a ${C}^{0,1}$ functions $f$ and a continuous weak Dirichlet process $X$ with finite quadratic variation,
	the decomposition \eqref{eq: intro ito} still holds true for some orthogonal  (or zero energy) process $\Gamma^f$.
	In particular, if $X$ and $f(\cdot,X)$ are both continuous local martingales,
	the orthogonal process  $\Gamma^f$ must vanish,
	so that \eqref{eq: intro ito} reduces to
	\begin{equation} \label{eq: intro replication}
		f(t, X_t)
		~=~
		f(0, X_{0})
		+
		\int_{0}^{t} \nabla_{x} f(s, X_{s}) dX_s,~\mbox{a.s.}
	\end{equation}
	 This is typically the case in mathematical finance under the so-called no free lunch with vanishing risk property, see e.g.~\cite{delbaen2006mathematics}. Such a formula is obviously very useful in many situations where $C^{1,2}$ regularity is difficult to prove, or not true at all. 
	In particular, we refer to \cite{gozzi2006weak} for an application to a verification argument in a stochastic control problem.

	\vspace{0.5em}

	In this paper, our first main objective is to provide an extension of  \eqref{eq: intro ito} and \eqref{eq: intro replication} to the functional (path-dependent) case. 	
	For $\mathbb{C}^{1,2}$-functionals, the functional It\^o's formula 
	 for continuous semimartingales
	has been investigated in \cite{cont2013functional, cosso2014regularization}, using the notion of Dupire's  \cite{dupireito} derivatives.
	For less regular functionals, a step forward in this direction was made in 
	 \cite{saporito2018functional, BT19, BT20a}. 
	The results in \cite{BT19, BT20a} were motivated, respectively, by  a verification argument for the replication of  path-dependent options in a model with market impact and  by an optional decomposition theorem for supermartingales, which in turn was applied to derive original results in the field  of robust hedging in mathematical finance. In the above papers, the functional does not even need to be differentiable in space but is assumed to be concave  in space and non-increasing in time (in a sense that matches the notion of Dupire's derivative), up to a smooth function. These assumptions, which perfectly match the cases of application motivating  \cite{BT19,BT20a}, allows one to show that $\Gamma^{f}$ is non-increasing without complex analysis.   
	It is   restricted to c\`adl\`ag semimartingales in \cite{BT20a} and to continuous semimartingales in \cite[Appendix]{BT19}.  
	 The main objective of \cite{saporito2018functional} is to establish a path-dependent Meyer-Tanaka's formula.
	It has the advantage over \cite{BT19,BT20a} to provide an explicit expression of the non-decreasing process entering the decomposition in terms of local times, but it requires much more regularity.
	
	 {Also} notice that a weaker notion of differentiability of path-dependent functionals has been used in  \cite{ekren2014viscosity, ekren2016viscosity} to define the viscosity solutions of path-dependent PDEs (see also \cite{ren2014overview} for an overview).

	\vspace{0.5em}

	In this paper, we show that the arguments of \cite{gozzi2006weak} can be used to easily  provide a functional version \eqref{eq: intro ito}-\eqref{eq: intro replication} using Dupire's notion of derivatives
	  for functionals $F$ defined on the space of paths. Unlike \cite{bandini2017weak}, we voluntarily restrict ourselves to the case where $X$ has continuous paths for tractability, see Remark \ref{rem : si X est pas continu}. 	Since we will use the stochastic calculus by regularization developed by Russo and Vallois, and their co-authors, 
	we naturally provide a version for weak Dirichlet processes that extends \cite{gozzi2006weak} to the path-dependent case. 
	In general, it requires additional conditions involving both the path-regularity of 
	  the underlying process $X$ and of the path-dependent functional $F$, that are satisfied when $X$ is a continuous semimartingale and $F$ is smooth, 
	or under other typical structure conditions on $F$, in particular if $F$ is Fr\'echet differentiable in space.

	\vspace{0.5em}
	
	Our main motivation comes from mathematical finance.
	In models without frictions, the prices of financial assets turn out to be semimartingales and even martingales under a suitable probability measure.
	The $\mathbb{C}^{0,1}$--functional It\^o's formula allows one to understand the structure/relation between the martingale parts of different financial assets,
	which is  the core problem for the hedging of risks.
	More concretely, we provide a new result on the super-hedging of path-dependent options, under model uncertainty,
	where the gradient of the value function provides the optimal super-hedging strategy.
	Unlike in \cite{BT20a}, the situation we consider does not correspond to that of a concave functional, so that the results of \cite{BT19,BT20a} can not be exploited. 
	In particular, in this application, the regularity on the first order derivative is proved by using PDE techniques, and seems to be original. 

	\vspace{0.5em}

	The rest of the paper is organized as follows.
	We first provide our version of the path-dependent It\^o's formula for $\mathbb{C}^{0,1}$-functionals and continuous weak Dirichlet processes in Section \ref{sec:Ito}.
	Applications in finance are then provided in Section \ref{sec:Applications}.

\section{Path-dependent It\^{o}'s formula for $\Cb^{0,1}$-functionals}
\label{sec:Ito}

	In this section, we fix a completed probability space $(\Om, \Fc, \P)$, equipped with a filtration $\F = (\Fc_t)_{t \in [0,T]}$ satisfying the usual conditions.
	The abbreviation u.c.p. denotes the uniform convergence in probability.

\subsection{Preliminaries}

	We start with preliminaries on the stochastic calculus via regularization and the notion of Dupire's derivatives of path-dependent functions.

\subsubsection{It\^o's calculus via regularization and weak Dirichlet processes}

	Let us recall here some definitions and facts on the It\^o calculus via regularization developped by Russo and Vallois \cite{russo1993forward, russo1995generalized, russo2007elements}. 
	See also Bandini and Russo \cite{bandini2017weak} (and \cite{russo1996ito, gozzi2006weak}) for a version of the $C^1$--It\^o's formula.

	\begin{Definition} \label{def:integral}
		$\mathrm{(i)}$ Let $X$ be a real valued {c\`adl\`ag} process, and $H$ be a process with paths in $L^1([0,T])$ a.s. The forward integral of $H$ w.r.t. $X$ is defined by
		$$
			\int_0^t H_s ~d^-X_s 
			~:=~ 
			\lim_{\eps \searrow 0} \frac1\eps \int_0^t H_s \big( X_{(s+\eps) \wedge t} - X_s \big) ds,
			~~t \ge 0,
		$$
		whenever the limit exists in the sense of u.c.p.
		
		\vspace{0.5em}
		
		\noindent $\mathrm{(ii)}$ 
		Let $X$ and $Y$ be two real valued {c\`adl\`ag} processes. The co-quadractic variation $[X,Y]$ is defined by 
		$$
			[X,Y]_{t}
			~:=~ 
			\lim_{\eps\searrow 0} \frac1\eps \int_{0}^{t} (X_{(s+\eps)\wedge t}-X_{s})(Y_{(s+\eps)\wedge t}-Y_{s})ds,
			~~t\ge 0,
		$$ 
		whenever the limit exists in the sense  of u.c.p.
		
		\vspace{0.5em}
		
		\noindent $\mathrm{(iii)}$ We say that a real valued {c\`adl\`ag} process $X$ has finite quadratic variation,
		if its quadratic variation, defined by $[X]:=[X,X]$, exists and is finite a.s.
	\end{Definition}

	\begin{Remark}
		When $X$ is a {(c\`adl\`ag)} semimartingale and $H$ is a  c\`adl\`ag adapted process,   $\int_0^t H_s ~d^-X_s$ coincides with the usual It\^o's integral $\int_0^t H_{s} dX_s$.
	 	When $X$ and $Y$ are two semimartingales,   $[X, Y]$ coincides with the usual bracket.
 	\end{Remark}

	\begin{Definition} \label{def:weakDirichlet}
		$\mathrm{(i)}$ We say that an adapted process $A$ is orthogonal if $[A, N] = 0$ for any real valued continuous local martingale $N$.
		
		\vspace{0.5em}
		
		\noindent $\mathrm{(ii)}$ An adapted   process $X$ is called a   weak Dirichlet process if it has a decomposition of the form 
		 $X = X_0 + M +A$, where $M$ is a local martingale and $A$ is orthogonal such that $M_0 = A_0 = 0$.
	\end{Definition}
	
	\begin{Remark}
		$\mathrm{(i)}$ An adapted  process with finite variation is orthogonal.
		Consequently, a   semimartingale is in particular a continuous weak Dirichlet process.

		\vspace{0.5em}
		
		\noindent $\mathrm{(ii)}$ An orthogonal process has not necessarily finite variation. For example, any   deterministic process (with possibly infinite variation) is orthogonal.

		\vspace{0.5em}
		
		\noindent $\mathrm{(iii)}$	
		The decomposition  $X= X_0 + M+A$ for a continuous weak Dirichlet process $X$ is unique,
		 and both processes $M$ and $A$ in the decomposition are continuous.
 	\end{Remark}

\subsubsection{Dupire's derivatives of path-dependent functions}

	Let us denote by $C([0,T])$ the space of all $\R^d$--valued continuous paths on $[0,T]$,
	and by $\DT$ the space of all $\R^d$--valued c\`adl\`ag paths on $[0,T]$,
	which are endowed with the uniform convergence topology induced by the norm $\Ninfty{\xr} := \sup_{s \in [0,T]}|\xr_s|$.
	Let $\Theta := [0,T] \x \DT$.
	For $(t,\xr)\in \Theta$, let us define the (optional) stopped path $\xr_{t \wedge}:=(\xr_{t\wedge s})_{s\in [0,T]}$. 	

	\vspace{0.5em}

	A function $F: \Theta \longrightarrow \R$ is said to be non-anticipative if $F(t, \xr)=F(t, \xr_{t \wedge})$ for all $(t,\xr)\in \Theta$.
	A non-anticipative function $F: \Theta \longrightarrow \R$ is said to be continuous if, 
	for all $(t, \xr) \in \Theta$ and $\eps > 0$, there exists $\delta > 0$ such that
	$$
		|t - t'| + \| \xr_{t \wedge } - \xr'_{t' \wedge } \| \le \delta
		~\Longrightarrow~
		|F(t, \xr) - F(t', \xr') | \le \eps.
	$$
	Let $\Cb(\Theta)$ denote the class of all non-anticipative continuous functions.
	A non-anticipative function $F$ is said to be left-continuous if,
	for all $(t, \xr) \in \Theta$ and $\eps > 0$, there exists $\delta > 0$ such that
	$$
		t' \le t, ~ |t - t'| + \| \xr_{t \wedge } - \xr'_{t' \wedge } \| \le \delta
		~\Longrightarrow~
		|F(t, \xr) - F(t', \xr') | \le \eps.
	$$
	We denote by $\Cb_{l}(\Theta)$ the class of all non-anticipative left-continuous functions.
	
	\vspace{0.5em}

	Let $F: \Theta \longrightarrow \R$ be a non-anticipative function, 
	we follow Dupire \cite{dupireito} to define the Dupire's derivatives:
	$F$ is said to be horizontally differentiable if,
	for all $(t,\xr)\in [0,T) \x \Theta$, its horizontal derivative
	$$
		\partial_{t}F(t,\xr) ~:=~ \lim_{h\searrow 0} \frac{F(t+h,\xr_{t \wedge }) - F(t,\xr_{t \wedge })}{h}
	$$
	is well-defined ;
	$F$ is said to be vertically differentiable if, for all $(t,\xr)\in \Theta$, the function
	$$
		y \longmapsto F(t,\xr\oplus_{t} y) 
		~\mbox{is differentiable at}~0, ~\mbox{with}~
		\xr\oplus_{t} y  
		~:=~
		\xr \1_{  [0,t)}  + (\xr_t + y) \1_{  [t,T]},
	$$
	whose derivative at $y=0$ is called the vertical derivative of $F$ at $(t, \xr)$, denoted by $\nabla_{\xr}F(t,\xr)$.
	One can then similarly define the second-order derivative $\nabla^{2}_{\xr}F$.
	 {Given}
	$$
		\Cb^{0,1}(\Theta) 
		:= 
		\big\{ 
			F \in \Cb(\Theta) ~:  \nabla_{\xr}F ~\mbox{is well defined and}~
			\nabla_{\xr}F \in \Cb_l(\Theta) 
		\big\},
	$$
	 {we let} $\Cb^{1,2}(\Theta)$ denote the class of all functions $F \in \Cb^{0,1}(\Theta)$ such that 
	both $\partial_t F$ and $\nabla^{2}_{\xr}F$ are well defined and  belong to $\Cb_l(\Theta)$.

 	\vspace{0.5em}

	A functional $F: \Theta \longrightarrow \R$ is said to be locally bounded if, for all $K > 0$,
	\begin{equation} \label{eq:locally_bounded}
		\sup_{t \in [0,T], ~\|\xr\| \le K} ~ \big|F(t, \xr) \big|   ~<~ \infty.
	\end{equation}
	Further, $F$ is said to be locally uniformly continuous if,   
	for each $K > 0$, there exists a modulus of continuity\footnote{ {A non-negative function that is  continuous at $0$ and vanishes at $0$}.} $\delta_K$ such that, 
	for all $t \in[0,T]$, $h \in [0,T-t]$, $\|\xr\| \le K$, $|y | \le K $,
	\begin{equation} \label{eq:unif_continuity_DF}
		\big| F(t, \xr) - F(t+h, \xr_{t \wedge }) \big| + \big| F(t, \xr) - F(t, \xr \oplus_{t} y) \big|
		~\le~
		\delta_K(h + |y|).
	\end{equation}
	Let us denote by $\Cb^{u,b}_{\rm loc} (\Theta)$ the class of  all locally bounded and locally uniformly continuous functions $F: \Theta \longrightarrow \R$. 
	Notice that a continuous function defined on $[0,T] \x \R^d$ is automatically locally bounded and locally uniformly continuous, 
	while it may not be true for a continuous function defined on $\Theta$.
	This is the reason for introducing the class $\Cb^{u,b}_{\rm loc}(\Theta)$.

	\vspace{0.5em}

	In the following, given a non-anticipative function $F: \Theta \longrightarrow \R$, we shall often write $F_t(\xr)$ in place of $F(t,\xr)$ for ease of notations.

\subsection{Functional It\^{o}'s formula for $\Cb^{0,1}(\Theta)$-functions}

 	We first provide a functional It\^o's formula for continuous weak Dirichlet processes. 
	More precisely, let $F \in \Cb^{0,1}$ and $X = M+A$ be a continuous weak Dirichlet process, 
	we give {a} necessary and sufficient condition for the following decomposition:
	\begin{equation} \label{eq:C01Ito}
		F(t,X) 
		~=~
		F(0,X)
		~+~ \int_{0}^{t} \nabla_{\xr} F_s(X) dM_{s}   
		~+~ \Gamma^F_t,
		~~~t \in [0,T],
	\end{equation}
	where $\Gamma^F$ is a continuous orthogonal process.

	\begin{Theorem}\label{thm: Ito} 
		Let $X = X_0 + M+A$ be a continuous weak Dirichlet process with finite quadratic variation,
		where $M$ is a (continuous) local martingale and $A$ is an orthogonal process.
		Let $F \in \Cb^{0,1}(\Theta)$ be such that both $F $ and $\nabla_{\xr} F$ belong to $\Cb^{u,b}_{\rm loc}(\Theta)$,  {and assume that $s\mapsto\nabla_{\xr} F_{s}(X)$ admits right-limits a.s.}		
		Then, $F(\cdot, X)$ is a continuous weak Dirichlet process with decomposition \eqref{eq:C01Ito} if and only if, for all continuous martingale $N$,
		\begin{equation} \label{eq:CNS_C01}
			\frac{1}{\eps} \int_{0}^{\cdot} \!\!\! \big(F_{s+\eps}(X) - F_{s+\eps}(X_{s \wedge }\oplus_{s+\eps}(X_{s+\eps}-X_{s})) \big) \big( N_{s+\eps}-N_{s} \big) ds
			\longrightarrow 0,
			~\mbox{u.c.p.,~as}~
			\eps \longrightarrow 0.
		\end{equation}
	\end{Theorem}
	
	The proof of Theorem \ref{thm: Ito} is postponed to the end of this section. 
	{N}otice that, when $F(t,\xr)=F_{\circ}(t,\xr_{t})$ for some $F_{\circ}: [0,T] \x \R^d \longrightarrow \R$,
	it is clear that $F_{s+\eps}(X) = F_{s+\eps}(X_{s \wedge }\oplus_{s+\eps}(X_{s+\eps}-X_{s}))$ so that \eqref{eq:CNS_C01} holds always true.
	Let us also provide a sufficient condition to ensure \eqref{eq:CNS_C01}.

	\begin{Proposition}\label{prop: cn pour cns}
		Assume that
		\begin{align}\label{eq: cond J2}
			E^{\eps}
			:=
			\int_{0}^{T}
			\frac{1}{\eps} \Big(F_{s+\eps}(X)-F_{s+\eps} \big(X_{s \wedge }\oplus_{s+\eps}(X_{s+\eps}-X_s) \big) \Big)^2
			ds
			\longrightarrow 0,\; \mbox{in probability}
		\end{align}
		 { as $\eps \longrightarrow 0$.} 
		Then, condition \eqref{eq:CNS_C01} holds true.
	\end{Proposition}
	\proof Using Cauchy-Schwarz inequality, it follows that, for all continuous martingale $N$,
	\begin{eqnarray*}
		&&
		\Big| \int_{0}^{\cdot} \frac{\left(F_{s+\eps}(X)-F_{s+\eps}(X_{s \wedge }\oplus_{s+\eps}(X_{s+\eps}-X_{s})) \right)}{\sqrt{\eps}}
		~\frac{N_{s+\eps}-N_{s}}{\sqrt{\eps}} ds \Big| \\
		&\!\!\!\le\!\!\!&
		\Big( \! \int_{0}^{\cdot}\! \frac{ \big(F_{s+\eps}(X)-F_{s+\eps}(X_{s \wedge }\oplus_{s+\eps}(X_{s+\eps}-X_{s})) \big)^2}{\eps} ds \Big)^{1/2}
		\! \Big( \! \int_0^{\cdot}\! \frac{(N_{s+\eps}-N_{s})^2}{\eps} ds \Big)^{1/2},
	\end{eqnarray*}
	which converges to $0$ in the sense of u.c.p. by \eqref{eq: cond J2}, together with the fact that $N$ has finite quadratic variation.
	\qed

	\begin{Remark}\label{rem: thm: Ito}
		The sufficient condition \eqref{eq: cond J2} is still quite abstract,
		we will provide more discussions on it in Section \ref{subsec:DiscussionsConditions}.
		In particular it is satisfied when $X$ is a continuous semimartingale and $F\in \Cb^{1,2}(\Theta)$, 
		so that the result in Theorem \ref{thm: Ito} is consistent with that in \cite{cont2013functional}. 
		Let us also notice that, to prove \eqref{eq:C01Ito}, it is indeed enough to check that for any sequence $(\eps_n)_{n \ge 1}$, such that $\eps_n \longrightarrow 0$,
		there exists a subsequence $(\eps_{n_k})_{k \ge 1}$ along which the convergence in \eqref{eq:CNS_C01} holds true.		
	\end{Remark}

	We next provide a direct consequence of Theorems \ref{thm: Ito}, by  {combining} it  {with} the Doob-Meyer decomposition, in the case where $X$ is a continuous martingale and $F(\cdot, X)$ is a supermartingale.
	Notice that, in the following context, our result is more precise than the classical Doob-Meyer decomposition for supermartingales.

	\begin{Corollary} \label{coro:DoobMeyer}
		Let $F: \Theta \longrightarrow \R$ satisfy the conditions in Theorem \ref{thm: Ito}.
		Assume in addition that $X$ is a continuous local martingale, and $F(\cdot, X)$ is a supermartingale.
		Then
		$$
			F(t,X) ~=~
			F(0, X) ~+ \int_0^t \nabla_{\xr} F_s (X) dX_s
			~+~ A_t,
			~~\mbox{for all}~t \in [0,T],
		$$
		where $A$ is a predictable non-increasing process. 
	\end{Corollary}
	\proof
		 
		It follows  {from} Theorem \ref{thm: Ito} 
		that the continuous supermartingale $F(\cdot, X)$ has the   decomposition
		\begin{equation} \label{eq:DecompFX}
			F(t,X) ~=~
			F(0, X) + \int_0^t \nabla_{\xr} F_s (X) dX_s
			+\Gamma^F_t,
		\end{equation}
		where $\Gamma^F$ is a continuous (predictable) orthogonal process.
		At the same time, $\Gamma^F$ should be a supermartingale, since $F(\cdot, X)$ is a supermartingale and $ \int_0^{\cdot} \nabla_{\xr} F_s (X) dX_s$ is a local martingale.
		Then, $\Gamma^F$ has finite variation, and hence
		\eqref{eq:DecompFX} coincides with the Doob-Meyer decomposition of $F(\cdot, X)$. As a conclusion, 
		 $\Gamma^F =  A$ for some predictable non-increasing process $A$.
	\qed
 
	\vspace{0.5em}

\noindent {\bf Proof of Theorem \ref{thm: Ito}.}
	Notice that $F \in \Cb(\Theta)$, $\nabla_{\xr}F \in \Cb_l(\Theta)$ and $X$ is a continuous process.
	Then the process $t \mapsto F_t(X)$ has a.s.~continuous paths,
	$t \mapsto \nabla_{\xr}F_t(X)$ has a.s.~left-continuous paths (see e.g. \cite[Lemma 2.6]{cont2013functional}).
	We now follow the arguments of \cite[Theorem 5.15]{bandini2017weak} to 
	show that \eqref{eq:CNS_C01} is a necessary and sufficient condition for the decomposition   \eqref{eq:C01Ito}.

	\vspace{0.5em}	
	
	\noindent $\mathrm{(i)}$ 
	Let us define the process $\Gamma^F$ by
	$$
		\Gamma^F_{\cdot} ~:=~ F_{\cdot}(X) - \int_{0}^{\cdot} \nabla_{\xr} F_{s}(X) dM_{s}.
 	$$
	We need to show that  {the} condition \eqref{eq:CNS_C01} is   necessary and sufficient   to ensure that 
	$\Gamma^F$ is an orthogonal process (Definition \ref{def:weakDirichlet}),
	that is, for any continuous local martingales $N$,
	$$
		\big[ \Gamma^F, N \big] 
		~=~
		\Big[ F_{\cdot}(X) - \int_{0}^{\cdot} \nabla_{\xr} F_{s}(X) dM_{s}  
		~,~
		N 
		\Big]
		~=~
		0.
	$$
	We first notice that, by \cite[Proposition 2.8]{bandini2017weak},
	$$
		\Big[ \int_{0}^{\cdot} \nabla_{\xr} F_{s}(X) dM_{s}, N \Big] 
		~=~
		\int_0^{\cdot} \nabla_{\xr} F_s(X) d [M, N]_s
		~=~
		\int_0^{\cdot} \nabla_{\xr} F_s(X) d [M, N]_s.
	$$
	Then, to prove the decomposition \eqref{eq:C01Ito}, it is equivalent to show that, for any continuous local martingale $N$,
	\begin{equation}\label{eq: prove equality convariations}
		I^{\eps}_{\cdot} := \frac{1}{\eps} \int_{0}^{\cdot} \!\! \big(F_{s+\eps}(X)-F_{s}(X)\big) \big( N_{s+\eps}-N_{s} \big) ds 
		\longrightarrow
		\int_{0}^{\cdot}\!\! \nabla_{\xr}F_{s}(X)d[M,N]_{s},
		~\mbox{as}~ 
		\eps\searrow 0, ~\mbox{u.c.p.},
	\end{equation}
	Let us write $I^{\eps} = I^{1, \eps} + I^{2, \eps}$, with
	$$
		I^{1,\eps}_{t} ~:=~ \frac{1}{\eps} \int_{0}^{t} \big(F_{s+\eps}(X_{s \wedge  }\oplus_{s+\eps}(X_{s+\eps}-X_{s}))- F_{s}(X) \big) \big( N_{s+\eps}-N_{s} \big)  ds,
	$$
	and
	$$
		I^{2,\eps}_{t}~:=~ \frac{1}{\eps} \int_{0}^{t} \big(F_{s+\eps}(X) - F_{s+\eps}(X_{s \wedge }\oplus_{s+\eps}(X_{s+\eps}-X_{s})) \big) \big( N_{s+\eps}-N_{s} \big) ds.
	$$

	\noindent $\mathrm{(ii)}$ 
	Let us first consider $I^{1,\eps}$ and write it as $I^{1,\eps} = I^{11,\eps}+I^{12,\eps} + I^{13,\eps} + I^{14,\eps}$, where
	$$
		I^{11,\eps}_{t} := \frac{1}{\eps} \int_{0}^{t}  \big( F_{s+\eps}(X_{s \wedge }) - F_{s}(X) \big)\! \big( N_{s+\eps} - N_{s} \big) ds,
	$$
	$$
		I^{12,\eps}_{t} := \frac{1}{\eps} \int_{0}^{t} \Delta^{\eps}_{s} \cdot \big(X_{s+\eps} - X_{s} \big) \big(N_{s+\eps} - N_{s} \big) ds,
	$$
	with
	$$
		\Delta^{\eps}_{s}
		:=
		\int_{0}^{1} 
			\Big(
				\nabla_{\xr} F_{s+\eps} \big( X_{s \wedge }\oplus_{s+\eps} \lambda( X_{s+\eps}-X_{s}) \big)
				-
				\nabla_{\xr}F_{s+\eps} (X_{s \wedge } ) 
			\Big)
		d\lambda.
	$$
	and
	$$
		I^{13,\eps}_{t} 
		:=
		\frac{1}{\eps} \int_{0}^{t} \big( \nabla_{\xr} F_{s+\eps} (X_{s \wedge } ) -  \nabla_{\xr} F_{s} (X_{s \wedge } ) \big) \cdot \big(X_{s+\eps}-X_{s} \big) \big(N_{s+\eps}-N_{s} \big) ds,
	$$
	$$
		I^{14,\eps}_{t} := \frac{1}{\eps} \int_{0}^{t} \nabla_{\xr} F_{s} (X_{s \wedge }) \cdot \big(X_{s+\eps}-X_{s} \big) \big(N_{s+\eps}-N_{s} \big) ds.
	$$

	For the term $I^{11,\eps}$, one has, by the integration by parts formula,
	$$
		I^{11,\eps}_{t}
		~=~
		\frac{1}{\eps} \int_{0}^{t} \left( \big(F_{s+\eps}(X_{s \wedge })-F_{s}(X) \big)  \int_{s}^{s+\eps}dN_{u} \right) ds
		~=~
		\int_{0}^{t+\eps} \theta^{11,\eps}_u dN_{u},
	$$
	where, by the uniform continuity condition \eqref{eq:unif_continuity_DF} on $F$,
	$$
		\theta^{11, \eps}_u 
		~:=~
		\frac{1}{\eps}  \int_{ (u-\eps)\vee 0}^{u} 
			\big( F_{s+\eps}(X_{s \wedge })-F_{s}(X) \big)
		ds 
		~\longrightarrow~ 
		0, ~\mbox{for all}~u \in [0,T],
		~\mbox{a.s.}
	$$
	Then, by e.g.~\cite[Theorem I.4.31]{jacod2013limit},
	$$
		I^{11,\eps}
		\longrightarrow
		0,~\mbox{u.c.p.}
		~\mbox{as}~
		\eps \longrightarrow 0.
	$$

	For the terms $I^{12, \eps}$ and $I^{13, \eps}$, we notice that
	$$
		\sup_{t \in [0,T]} \big( \big| I^{12,\eps}_t  \big| + \big| I^{13,\eps}_t  \big| \big)
		~\le~
		\delta_{\eps} ~ \Big( \frac{1}{\eps} \int_0^T \big|X_{s+\eps}-X_{s} \big|^2 ds \Big) \Big( \frac{1}{\eps} \int_0^T \big(N_{s+\eps}-N_{s} \big)^2 ds \Big),
	$$
	where
	$$
		\delta_{\eps} 
		:= 
		\sup_{0 \le s \le T-\eps} 
		\Big( |\Delta^{\eps}_s| + \big| \nabla_{\xr} F_{s+\eps} (X_{s \wedge } ) -  \nabla_{\xr} F_{s} (X_{s \wedge } ) \big| \Big)
		\longrightarrow 0,
		~\mbox{a.s. as}~
		\eps \searrow 0,
	$$
	by the uniformly continuity condition \eqref{eq:unif_continuity_DF} on $\nabla_{\xr}F$.
	Since 
	$$
		\Big( \frac{1}{\eps} \int_0^T \big(X_{s+\eps}-X_{s} \big)^2 ds \Big) \Big( \frac{1}{\eps} \int_0^T \big(N_{s+\eps}-N_{s} \big)^2 ds \Big)
		~\longrightarrow~[X]_T [N]_T,~\mbox{u.c.p. as}~ 
		\eps \longrightarrow 0,
	$$
	it follows that $I^{13,\eps} \longrightarrow~ 0$ and $I^{13,\eps} \longrightarrow~ 0$, u.c.p.

	\vspace{0.5em}

	Finally, for $I^{14,\eps}$, we apply \cite[Corollary A.4 and Proposition A.6]{bandini2017weak} to obtain that
	$$
		I^{14, \eps}_{\cdot} \longrightarrow \int_0^{\cdot} \nabla_{\xr} F_s(X) d[M, N]_s,
		~\mbox{u.c.p., as}~
		\eps \longrightarrow 0,
	$$
	so that 
	$$
		I^{1, \eps}_{\cdot} \longrightarrow \int_0^{\cdot} \nabla_{\xr} F_s(X) d[M, N]_s,
		~\mbox{u.c.p., as}~
		\eps \longrightarrow 0.
	$$
	
	\noindent $\mathrm{(iii)}$ 	
	To conclude, we observe that \eqref{eq: prove equality convariations} holds true (or equivalently \eqref{eq:C01Ito} holds) if and only if $I^{2, \eps} \longrightarrow 0$, u.c.p.~(or equivalently \eqref{eq:CNS_C01} holds),
	for any continuous local martingale $N$.
	This concludes the proof.
\endproof

	\begin{Remark}\label{rem : si X est pas continu} 
  	The results and proof of Theorem \ref{thm: Ito}
	remain valid even if $X$ is a c\`adl\`ag weak Dirichlet process with bounded quadratic variation, up to the fact that $\nabla_{\xr} F_s (X)$ must be replaced by $\nabla_{\xr} F_s (X\1_{[0,s)}+X_{s-}\1_{[s,T]})$ in \eqref{eq:C01Ito}.
	However, in this case, 
	the decomposition of the weak Dirichlet process $F_{\cdot}(X)$ may not be unique.
	To see this, recall that any purely discontinuous martingale is orthogonal to a continuous martingale $N$,
	then one can always move a purely discontinuous martingale from the martingale part of $F_{\cdot}(X)$ to the orthogonal part of $F_{\cdot}(X)$ and the decomposition of the weak Dirichlet process $F_{\cdot}(X)$ stays valid.
	
	\vspace{0.5em}
	
	To ensure the uniqueness of the decomposition of $F_{\cdot}(X)$, 
	one needs to use the notion of special weak Dirichlet process in \cite{bandini2017weak},
	where the orthogonal part $\Gamma^{F}$ entering \eqref{eq:C01Ito} is required to be predictable.
	Then it is possible to mimic the smoothing procedure of \cite{bandini2017weak} to obtain such a decomposition for $F_{\cdot}(X)$.
	However, smoothing a $\Cb^{0,1}(\Theta)$-function into a $\Cb^{1,2}(\Theta)$-function 
	requires various and heavy technical assumptions, see \cite{saporito2018functional}, 
	which may be difficult to check in the applications we have in mind. 	
	\end{Remark}

\subsection{Discussions on the condition \eqref{eq: cond J2} }
\label{subsec:DiscussionsConditions}

	The sufficient technical condition \eqref{eq: cond J2} used  to ensure the decomposition result in Theorem  \ref{thm: Ito} is still too abstract.
	Let us provide some {more}  explicit sufficient conditions for  \eqref{eq: cond J2}.
	We first show that  \eqref{eq: cond J2} holds true when $X$ is a continuous semimartingale and  $F \in \mathbb{C}^{1,2}(\Theta)$, which makes our result consistent with \cite{cont2013functional}.
	We will then provide some examples of sufficient conditions for \eqref{eq: cond J2} when $F$ is not in $\mathbb{C}^{1,2}(\Theta)$.
	Also recall that \eqref{eq: cond J2}    trivially holds  when $F$ is Markovian, i.e. $F(t,\xr)=F_{\circ}(t,\xr_{t})$  for all $(t,\xr) \in [0,T]\x \DT$, for some $F_{\circ}: [0,T] \x \R^d \longrightarrow \R$.

\subsubsection{The case where $F \in \mathbb{C}^{1,2}(\Theta)$ and $X$ is a continuous semimartingale}
\label{sec: cas C12}

	When $X$ is a continuous semimartingale and $F \in \mathbb{C}^{1,2}(\Theta)$ with (local) bounded and uniformly continuous derivatives, 
	one can check that \eqref{eq: cond J2}   holds true by simply applying the functional It\^o's formula of  \cite{cont2013functional}.

	\begin{Proposition}\label{prop: cas C2}
		Let $X$ be a continuous semimartingale and  $F  \in \mathbb{C}^{1,2}(\Theta)$ be such that 
		 $\partial_t F$ and $ \nabla^2_{\xr} F$ are locally bounded, and $\nabla_{\xr}F  \in \Cb^{u,b}_{\rm loc}$.
		Then, condition \eqref{eq: cond J2}  holds true.
	\end{Proposition}
	\proof
	For simplification of the notations, let us consider the one-dimensional case.
	First, for every fixed $(s, \eps)$, we apply the functional It\^o's formula in \cite[Theorem 4.1]{cont2013functional}
	on $F(X)$ to obtain that
	\begin{eqnarray*}
		&&
		F_{s+\eps}(X) - F_{s+\eps}(X_{s \wedge }) 
		~=~
		\big( F_{s+\eps}(X) - F_s(X) \big)  + \big( F_s(X) -  F_{s+\eps}(X_{s \wedge }) \big) \\
		&=\!\!\!&
		\int_s^{s+\eps} \!\!\! \big( \partial_t F_r(X) - \partial_t F_r(X_{s \wedge }) \big) dr
		+
		\int_s^{s+\eps} \!\!\! \nabla_{\xr}F_{{r}}(X )  d X_r 
		+
		\frac12 \int_s^{s+\eps} \!\!\! \nabla^2_{\xr} F_{{r}}(X )  d [X]_r.
	\end{eqnarray*}
	Further, one can also apply the classical It\^o's formula {to} $\phi(X_r) := F_{s+\eps}(X_{s \wedge } \oplus_{s+\eps}(X_r - X_s))$ to obtain that
	\begin{eqnarray*}
		\!\!\!\!\! && \!\!\!\!\! F_{s+\eps} \big(X_{s \wedge }  \oplus_{s+\eps} (X_{s+\eps} - X_s) \big) - F_{s+\eps} \big(X_{s \wedge } \big)
		=
		\int_s^{s+\eps} \!\!\! \nabla_{\xr}F_{s+\eps}(X_{s\wedge } \oplus_{s+\eps} (X_r - X_s))  d X_r  \\
		&&~~~~~~~~~~~~~~~~~~~~~~~~~~~~~~~~~~~~~~~~+~ 
		\frac12 \int_s^{s+\eps} \!\! \nabla^2_{\xr} F_{s+\eps}(X_{s \wedge } \oplus_{s+\eps} (X_r - X_s))  d [X]_r.
	\end{eqnarray*}	
	Then, it follows that
	\begin{equation} \label{eq:F_diff_eps} 
		F_{s+\eps}(X) \!-\! F_{s+\eps} \big(X_{s \wedge }  \oplus_{s+\eps} (X_{s+\eps} - X_s) \big) 
		=
		\int_s^{s+\eps} \!\!\! W^{1,\eps}_{s,r} dr 
		+\!\! 
		\int_s^{s+\eps} \!\!\! W^{2,\eps}_{s,r} dX_r 
		+ \!
		\frac12 \! \int_s^{s+\eps} \!\!\! W^{3, \eps}_{s,r} d[X]_r,
	\end{equation}
	where
	$$
		W^{1,\eps}_{s,r} := \partial_t F_r(X) - \partial_t F_r(X_{s \wedge }),
		~~~
		W^{2,\eps}_{s,r} := \nabla_{\xr}F_r(X)
			\!-\! 
			\nabla_{\xr}F_{s+\eps} \big(X_{s\wedge } \oplus_{s+\eps} (X_r - X_s) \big),
	$$
	and
	$$
		W^{3,\eps}_{s,r} := \nabla^2_{\xr} F_r(X) -\nabla^2_{\xr} F_{s+\eps} (X_{s\wedge } \oplus_{s+\eps} (X_r - X_s)).
	$$
	By the local boundedness of $\partial_t F$, $\nabla_{\xr} F$ and $\nabla^2_{\xr}F$, it follows that
	$$
		\sup_{0 \le s \le r \le T, ~\eps > 0} 
		\Big( \big|W^{1,\eps}_{s,r}\big| + \big|W^{2,\eps}_{s,r}\big| + \big|W^{3,\eps}_{s,r}\big| \Big) 
		< 
		\infty,
		~\mbox{a.s.}
	$$
	Further, since $\nabla_{\xr}F$ satisfies  {the} (locally) uniform continuity condition \eqref{eq:unif_continuity_DF},
	for every fixed $r \in [0,T]$,
	one has
	$$
		\big|W^{2,\eps}_{s,r} \big|
		\le
		\big|
			\nabla_{\xr}F_r(X) - \nabla_{\xr}F_s(X)
		\big|
		+\delta_K(\eps {+ | X_r - X_s |}),
		~~\mbox{whenever}~\|X\| \le K. 
	$$
	As $s \mapsto \nabla_{\xr} F_s(X)$ is left-continuous,
	then for every fixed $r \in [0,T]$,
	\begin{equation} \label{eq:W2eps}
		\sup_{s \in [(r-\eps)\vee 0, r]} \big|W^{2,\eps}_{s,r} \big|  \longrightarrow 0,
		~\mbox{a.s.}~
		~\mbox{as}~\eps \longrightarrow 0.
	\end{equation}
	Let $X = X_0 + M +  A$   {where M is a} continuous martingale  {and $A$ is a} finite variat{ion} process,
	 {and} denote by $(|A|_t)_{t \in [0,T]}$ the total variation process of $A$.
	The two non-decreasing processes ${|A|}$ and $[M]$ are continuous, so that {they} are uniformly continuous on $[0,T]$, a.s.
	Recall the definition of $E^{\eps}$ in \eqref{eq: cond J2}. It follows {from} \eqref{eq:F_diff_eps} that
	$$
		E^{\eps} 
		~\le~
		4 E^{\eps}_1  + 4 E^{\eps}_2  + 4 E^{ \eps}_3 + 4 E^{ \eps}_4 ,
	$$
	where
	\begin{eqnarray*}
		E^{ \eps}_1
		&:=&
		\frac{1}{\eps} \int_0^T \Big( \int_s^{s+\eps}  W^{1,\eps}_{s,r} dr \Big)^2 ds 
		\le
		\int_0^T \!
			\int_s^{s+\eps} \! \big| W^{1,\eps}_{s,r} \big|^2  d r
		ds
		~\longrightarrow~
		0, ~\mbox{a.s.},
\\	
		E^{ \eps}_2 
		&:=&
		\frac{1}{\eps} \int_0^T \Big( \int_s^{s+\eps}  W^{2,\eps}_{s,r}  dA_r \Big)^2 ds 
		~\le~
		\int_0^T \!
			\big(|A|_{s+\eps} - |A|_s\big)
			\frac{1}{\eps}  \int_s^{s+\eps} \! \big| W^{2,\eps}_{s,r} \big|^2  d |A|_r
		ds \\
		&=&
		\int_0^T 
		\Big( \frac{1}{\eps}  \int_{(r-\eps)\vee 0}^{r} \big(|A|_{s+\eps} - |A|_s\big) \big| W^{2,\eps}_{s,r} \big|^2 ds \Big)
		d |A|_r
		~\longrightarrow~
		0, ~\mbox{a.s.},
\\
		E^{\eps}_3
		&:=&
		\frac{1}{4 \eps} \int_0^T \Big( \int_s^{s+\eps}  W^{3,\eps}_{s,r}  d[X]_r \Big)^2 ds  \\
		&\le&
		\int_0^T 
		\Big( \frac{1}{\eps}  \int_{(r-\eps)\vee 0}^{r} \big([X]_{s+\eps} - [X]_s\big) \big| W^{3,\eps}_{s,r} \big|^2 ds \Big)
		d [X]_r
		~\longrightarrow~
		0, ~\mbox{a.s.},
	\end{eqnarray*}	
	and
	$$
		E^{\eps}_4 := 
		\frac{1}{\eps} \int_0^T \Big( \int_s^{s+\eps}  W^{2,\eps}_{s,r}   dM_r \Big)^2  ds.
	$$
	To study the limit of $E^{\eps}_4$, one can assume w.l.o.g.~that $W^{2,\eps}$ and $[M]_T$ are uniformly bounded by {using} localization techniques.
	Then, by \eqref{eq:W2eps},
	$$
		\E \big[ \big| E^{\eps}_4 \big| \big] 
		=
		\frac{1}{\eps} \E \Big[ \int_0^T \int_s^{s+\eps} \big| W^{2,\eps}_{s,r}  \big|^2 d [M]_r ds \Big]
		=
		\E \Big[ \int_0^T  \Big( \frac{1}{\eps}  \int_{(r-\eps)\vee 0}^{r} \big| W^{2,\eps}_{s,r} \big|^2 ds \Big) d [M]_r \Big]
		\longrightarrow
		0.
	$$
	It follows that $E^{ \eps} \longrightarrow 0$ in probability, and {therefore that} \eqref{eq: cond J2} holds true.
	\endproof

\subsubsection{Examples of sufficient conditions {for} \eqref{eq: cond J2}}
\label{sec: cond structure pour cond de regu addi}

	We now provide   examples of    sufficient conditions  for  \eqref{eq: cond J2}. 
	The general idea behind  {them} is to exploit Item (ii) of Definition \ref{def:integral} to control the terms in \eqref{eq: cond J2}   by some quadratic variations, possibly up to an additional vanishing element.
	In the following, we let $\mathrm{BV}_{+}$ denote the collection of all  non-decreasing paths on $[0,T]$. 
 
 	\begin{Proposition}\label{prop: cond pour J2 sans carre} 
  		 Assume that, for all $\xr\in \DT$, $s\in [0,T]$ and $\eps \in [0,T-s]$, 
		$$
			\big|F_{s+\eps}(\xr)- F_{s+\eps}(\xr_{s \wedge }\oplus_{s +\eps} (\xr_{s+\eps}-\xr_s)) \big|
			~\le~ 
			\int_{(s, s+\eps)} \phi \big(\xr, |\xr_{u} - \xr_{s }|\big)db_u (\xr),
		$$
		where  $\phi: \CT \x \R_+  \longrightarrow \R$ satisfies $ \sup_{|y|\le K} \phi(\xr,y )  <\infty$, $\lim_{y \searrow 0}  \phi(\xr,y ) = \phi(\xr, 0) =0$ for all $\xr \in \CT$ and $K > 0$,  
		and $b: \CT \longrightarrow {\rm BV}_{+}$.
		Then, \eqref{eq: cond J2}   holds for any continuous process $X$.	
	\end{Proposition}
	\begin{proof}
	We first notice that, {since} $X$ is a continuous process, then, for all $u \in [0,T]$,
	$$
		\frac1\eps \int_{(u-\eps)\vee 0}^{u}  \phi \big(X, |X_{u}-X_{s }|\big)^2 ds  \longrightarrow \phi \big(X, 0 \big)^2
		=0, 
		~\mbox{a.s., as}~
		\eps \longrightarrow 0.
	$$
	Recall the definition of $E^{\eps}$ in \eqref{eq: cond J2},   and define the process $B$ with finite variations by $B := b(X)$.
	{Then,} Minkowski's integral inequality {implies} that
	$$
		\sqrt{E^{\eps}} ~\le~ \int_0^T \Big( \frac{1}{\eps} \int_{(u-\eps)\vee 0}^{u}  \phi \big(X, |X_{u}-X_{s }|\big)^2 ds \Big)^{1/2} d B_u
		\longrightarrow   0, ~\mbox{a.s.},
	$$
	which concludes the proof.
	\end{proof}

	\begin{Example}  \label{exam:J2} 
		Assume that there exists a family of signed measures 
		$(\mu(\cdot;t,\xr), \;(t,\xr) \in [0,T]\x D([0,T])$, which is dominated by a non-negative finite measure $\hat \mu$, and a locally bounded map $m:[0,T]\x \DT\x \DT\mapsto \R$ 
		such that,
		for all $t\le T\;\mbox{ and }\xr,\xr'\in D([0,T])$ satisfying $\xr_{t}=\xr'_{t}$,
		\begin{align}\label{eq: conti vp}
			F(t,\xr)-F(t,\xr')
			=
			\int_{[0,t)} \big( \xr_{s}-\xr'_{s} \big) \mu(ds;t,\xr) +  o \big( \|\xr_{t\wedge }-\xr'_{t\wedge }\| \big)m(t,\xr,\xr').
		\end{align}
		Then, one has $\partial_{\lambda } F_{t}(\xr+\lambda(\xr'-\xr))=\int_{0}^{t} (\xr'_{s}-\xr_{s}){\mu(ds;t,\xr+\lambda(\xr'-\xr))}$. 
		It follows that
		\begin{align*}
			F_{s+\eps}(X)- F_{s+\eps}(X_{s \wedge }\oplus_{s +\eps} (X_{s+\eps}-X_s))
			&=\int_{0}^{1} \partial_{\lambda } F_{s+\eps}(X^{\eps,\lambda})d\lambda\\
			&=\int_{0}^{1} \int_{(s,s+\eps)} (X_{s}-X_{u})\mu(du;s+\eps,X^{\eps,\lambda}) d\lambda,
		\end{align*}
		with $X^{\eps}:=X_{s \wedge }\oplus_{s +\eps} (X_{s+\eps}-X_s)$ and $X^{\eps,\lambda}:=X^{\eps}+\lambda(X-X^{\eps})$. 
		As $(\mu(\cdot;t,\xr))_{t,\xr}$ is dominated by $\hat \mu$, letting $\hat b_u := \hat \mu([0,u])$, one has 
		$$
			\big |F_{s+\eps}(X)- F_{s+\eps}(X_{s \wedge }\oplus_{s +\eps} (X_{s+\eps}-X_s)) \big|
			~\le~
			\int_{(s,s+\eps)} |X_{s}-X_{u}|d \hat b_{u}.
		$$
		Then, \eqref{eq: cond J2} holds true when $X$ has continuous paths, 
		by Proposition \ref{prop: cond pour J2 sans carre}.
 
		\vspace{0.5em}

		Notice that a Fr\'echet differentiable function in the sense of Clark \cite{clark1970representation} satisfies \eqref{eq: conti vp}.
		The difference is that we only need to check  \eqref{eq: conti vp} for paths $\xr$ such that $\xr_{t}=\xr'_{t}$.
	\end{Example}

 	When $X$ is a semimartingale, we can also {exploit} its semimartigale property to obtain sufficient conditions for \eqref{eq: cond J2}.
 
	\begin{Proposition} \label{prop: cond pour J22}
		Assume that, for all $\xr \in \DT$, $s \in [0,T]$ and $\eps \in [0, T-s]$,
		$$
			\big| F_{s+\eps}(\xr)- F_{s+\eps}(\xr_{s \wedge }\oplus_{s +\eps} (\xr_{s+\eps}-\xr_{s})) \big|
			~\le~
			\phi\big(\xr,\| \xr_{ (s+\eps)\wedge }-  \xr_{s \wedge }\|, \eps \big),	
 		$$
		where $\phi: C([0,T] \x \R_+ \x \R_+ \longrightarrow \R$ satisfies
		$$
			\lim_{\eps \searrow 0, ~y \searrow 0}\sup_{\|\xr\|\le K} {| \phi(\xr,y, \eps)|}/{y} = 0,
			~\mbox{for all}~
			K\ge 0.
		$$
		Assume in addition that $X$ is a continuous semimartingale.
		Then \eqref{eq: cond J2} holds true.
	\end{Proposition}
	\proof We first notice that
	$$
		\int_{0}^T
		\frac{ \phi(X,\| X_{ (s+\eps) \wedge }- X_{s \wedge }\|, \eps)^2}{\| X_{(s+\eps)\wedge }- X_{s \wedge }\|^2}
		\frac{\| X_{ (s+\eps) \wedge }- X_{s \wedge }\|^2}{\eps} ds
		~\le~ 
		C^2_{\eps}  \int_{0}^T \frac{1}{\eps} \| X_{ (s+\eps) \wedge }- X_{s \wedge }\|^{2} ds,
	$$
	where 
	$$
		C_{\eps}
		:=
		\sup_{s \in [0,T]} \frac{ \phi(X,\|X_{(s+\eps)\wedge }-X_{s \wedge }\|, \eps)}{\| X_{ (s+\eps)\wedge }- X_{s\wedge }\|} 
		\longrightarrow 0,
		~~\mbox{as}~\eps \longrightarrow 0.
	$$
	Further, up to adding additional components, one can assume that each component of $X$ is a martingale or a non-decreasing process. 
	We therefore assume this and it suffices to consider the one dimensional case. 
	Since $X$ is a martingale or a non-decreasing process, there exists $C>0$ such that 
	$$
		\E \Big[ \!\int_{0}^T \!\! \frac{1}{\eps} \big\| X_{ (s+\eps) \wedge }- X_{s\wedge } \big\|^{2} ds \Big]
		\le 
		C  \E\Big[\! \int_{0}^T \!\! \frac{1}{\eps} \big| X_{s+\eps}- X_{s} \big|^{2}ds \Big] 
		\longrightarrow C  \E \big [ X \big]_T,
		~\mbox{in probability}.
	$$
	This is enough to prove that $E^{\eps} \longrightarrow 0$ in probability, so that \eqref{eq: cond J2} holds.
	\endproof

	By combining the conditions in Propositions \ref{prop: cond pour J2 sans carre} and \ref{prop: cond pour J22}, one obtains immediately new sufficient conditions for \eqref{eq: cond J2}.

	\begin{Corollary}
		Assume that $X = (X^1, X^2)$, where $X^1$ is a continuous process, $X^2$ is a continuous semimartingale,
		and, for all $0 \le s \le s + \eps \le T$ and $\xr = (\xr^1, \xr^2) \in \DT$,
		$$
			\big|F_{s+\eps}(\xr)- F_{s+\eps}(\xr_{s \wedge }\oplus_{s +\eps} (\xr_{s+\eps}-\xr_s)) \big|
			\le
			\int_{[s, s+\eps]} \!\!\!\!\! \phi_1 \big(\xr, |\xr_{u} - \xr_{s }|\big)db_u (\xr)
			+
			\phi_2 \big(\xr,\| \xr^2_{ (s+\eps)\wedge }-  \xr^2_{s \wedge }\|, \eps \big),
		$$
		where $(\phi_1, b)$ satisfies the conditions in Proposition \ref{prop: cond pour J2 sans carre}, and $\phi_2$ satisfies the conditions in Proposition \ref{prop: cond pour J22}.
		Then \eqref{eq: cond J2} holds true.
	\end{Corollary}

 \section{Applications in mathematical finance}
 \label{sec:Applications}

	In {frictionless financial models},
	under the no-arbitrage (in the sense of no free lunch with vanishing risk) {assumption},
	the prices of tradable financial assets need to be  semimartingales, see e.g.~\cite{delbaen2006mathematics}.
	If the pricing function of a financial derivative is $\mathbb{C}^{0,1}(\Theta)$, then one can apply the It\^o's formula in Theorems \ref{thm: Ito}  to {characterize} the martingale part of the derivative's price process,
	 and therefore   identify the hedging strategy.
	Below we provide  some examples of such applications in finance.

 \subsection{General formulations under $\Cb^{0,1}(\Theta)$-regularity condition}
 
 \subsubsection{Replication of path-dependent options}
 	{Let us consider a continuous martingale $X = (X_t)_{0 \le t \le T}$}, which represents the discount{ed} price of some risky asset, {and}
	  a path-dependent derivative with payoff $g(X)$ such that $\E[ | g(X) | ] < \infty$.
	Define
	\begin{equation} \label{eq:def_V_ClarkOcone}
		V(t,\xr) := \E \big[ g(X) \big| X_{t \wedge } = \xr_{t \wedge } \big],\;{(t,\xr)\in [0,T]\x D([0,T]).}
	\end{equation}

 	\begin{Proposition} \label{prop:replication}
		Assume that $V $ {belongs to} $\Cb^{0,1}(\Theta)$ {and} satisfies all the conditions {of} Theorem \ref{thm: Ito}.
		Then
		$$
			g(X) = \E \big[ g(X) \big] + \int_0^T \nabla_{\xr} V(t, X) dX_t.
		$$
	\end{Proposition}
	\proof
		{Since} $V(t,X)$ and $-V(t,X)$ are both supermartingales,   the result follows {from} Corollary \ref{coro:DoobMeyer}.
	\qed

 	\begin{Remark}
		$\mathrm{(i)}$ The {above} result {can} be compared to \cite[ Theorem 5.2 ]{cont2013functional} {but}
		  we require less regularity conditions ($\Cb^{0,1}(\Theta)$ and \eqref{eq:CNS_C01} rather than $\Cb^{1,2}_b(\Theta)${, which implies \eqref{eq:CNS_C01} by Propositions \ref{prop: cn pour cns} and \ref{prop: cas C2}}).
		
		\vspace{0.5em}
		
		\noindent $\mathrm{(ii)}$ Let $X$ be a  diffusion process with dynamics
		$$
			X_t = X_0 + \int_0^t \mu(s, X) ds + \int_0^t \sigma(s, X) dW_t,
		$$
		{in which $W$ is a} Brownian motion and $(\mu,\sigma)$ are continuous, non-anticipative and Lipschitz in space.
		When $V \in \Cb^{1,2}(\Theta)$ in the sense of \cite[Theorem 4.1]{cont2013functional},  
		it is easy to deduce from their functional It\^o's formula that $V$ is a classical solution of the path-dependent PDE
		$$
			\partial_t V + \mu \cdot \nabla_{\xr} V + \frac12 \sigma\sigma^{\top} \cdot \nabla^2_{\xr} V = 0.
		$$
		  Without the $\Cb^{1,2}(\Theta)$--regularity condition,
		one can still prove that $V$ is a viscosity solution of the path-dependent PDE in the sense of \cite{ekren2016viscosity},
		{for which} numerical  algorithms can be found in  \cite{ren2017convergence,zhang2014monotone}.
	\end{Remark}
 
 	\begin{Remark}
		As already mentioned in \cite{dupireito} and \cite{cont2013functional}, the result {of} Proposition \ref{prop:replication} is consistent with the classical Clark-Haussmann-Ocone formula.
		Indeed, let $X$ be a continuous martingale with independent increments, and $g$ {be} Fr\'echet differentiable with derivative $\lambda_{g}$, then by the Clark-Haussmann-Ocone formula (see e.g. Haussmann \cite{haussmann1978functionals}),
		$$
			g(X) = \E[g(X)] + \int_0^T \E \big[ \lambda_{g}(X; [t,T]) \big | \Fc_t \big] dX_t.
		$$
		On the other hand,  for the value function $V$ in \eqref{eq:def_V_ClarkOcone}, 
		one can also compute the vertical derivative $\nabla_{\xr} V$ from its definition to obtain that
		$$
			\nabla_{\xr} V(t,\xr) = \E \big[ \lambda_{g}(X; [t,T])  \big | X_{t \wedge } = \xr_{t \wedge } \big].
		$$
	\end{Remark}
 
\subsubsection{Super-replication under model uncertainty}

 	Let us now denote by $\Om^{\circ} := D([0,T] , \R^d)$ the canonical space of $\R^d$-valued c\`adl\`ag paths on $[0,T]$,
	let $X$ be the canonical process, and $\F^{\circ} = (\Fc^{\circ}_t)_{t \in [0,T]}$ the canonical filtration.
	Let us denote by $\mathfrak{B}(\Om^{\circ})$ the space of all Borel probability measures on $\Om^{\circ}$.
	We consider a subset $\Pc \subset \mathfrak{B}(\Om^{\circ})$, such that $X$ is a $\P$--continuous local martingale satisfying $\P[X_0 = x_0] = 1$ for all $\P \in \Pc${, for some $x_{0}\in \R^{d}$}.
	Recall that, given a probability measure $\P$ on $(\Om^{\circ}, \Fc^{\circ}_T)$ and a $\F^{\circ}$--stopping time $\tau$ taking values in $[0,T]$, a r.c.p.d. (regular conditional probability distribution) of $\P$ conditional to $\Fc^{\circ}_{\tau}$ is a family $(\P_{\om})_{\om \in \Om}$ of probability measures on $(\Om^{\circ}, \Fc^{\circ}_T)$, such that $\om \mapsto \P_{\om}$ is $\Fc^{\circ}_{\tau}$--measurable, $\P_{\om}[X_s = \om_s, ~s \le \tau(\om)] = 1$ for all $\om \in \Om$, and  $\E^{\P}[\1_A | \Fc^{\circ}_{\tau}](\om) = \E^{\P_{\om}}[\1_A]$ for $\P$--a.e. $\om \in \Om^{\circ}$ for all $A \in \Fc^{\circ}_T$.
	Recall also that a subset $A$ of a Polish space $E$ is called an analytic set if there exists another Polish space $E'$ together with a Borel subset $B \subset E \x E'$ such that $A = \{x \in E ~: (x, x') \in B \}$.
	
	We further make the following assumptions.
	
	\begin{Assumption} \label{assum:DPP}
		One has $\Pc = \cup_{\om \in \Om^{\circ}} \Pc(0, \om)$, for a collection of families of probability measures $\big(\Pc(t,\om) \big)_{(t, \om) \in [0,T] \x \Om^{\circ}}$ on $\Om^{\circ}$.
		Moreover, for every $(t, \om) \in [0,T] \x \Om^{\circ}$:
		\begin{enumerate}
			\item $\Pc(t, \om) = \Pc(t, \om_{t \wedge })$, $\P[X_{t \wedge } = \om_{t \wedge }] = 1$ for all $\P \in \Pc(t, \om)$ and the graph set 
			{$$
				[[\Pc]] ~:=~ \big\{(t, \om, \P) ~: \P \in \Pc(t, \om) \big\}
			$$}
			 is an analytic subset of $[0,T] \x \Om^{\circ} \x \mathfrak{B}(\Om^{\circ})$.

			\item Let $\P \in \Pc(t, \om)$, $s \ge t$ and $(\P_{\om})_{\om \in \Om^{\circ}}$ be a family of regular conditional probability of $\P$ knowing $\Fc^{\circ}_s$, then $\P_{\om} \in \Pc(s, \om)$ for $\P$-a.e. $\om \in \Om^{\circ}$.
			
			\item Let $\P \in \Pc(t, \om)$, $s \ge t$ and $(\Q_{\om})_{\om \in \Om^{\circ}}$ be a family such that 
			$\om \mapsto {\Q}_{\om}$ is $\Fc^{\circ}_s$--measurable and
			$\Q_{\om} \in \Pc(s, \om)$ for $\P$-a.e. $\om \in \Om^{\circ}$, then
			$$
				\P \otimes_s \Q_{\cdot} \in \Pc(t, \om),
			$$
			where $\P \otimes_s \Q_{\cdot} $ is defined by
			$$
				\E^{\P \otimes_s \Q_{\cdot}}[ \xi] := \int_{\Om^{\circ}} \int_{\Om^{\circ}} \xi(\om') \Q_{\om}(d \om') \P(d \om),
				~\mbox{for all bounded r.v.}~\xi: \Om^{\circ} \longrightarrow \R.
			$$
		\end{enumerate}
	\end{Assumption}

 	Let $g: \Om^{\circ} \longrightarrow \R$ be such that $\sup_{\P \in \Pc} \E^{\P} \big[ \big| g(X) \big| \big] < \infty$, let us define
	$$
		V(t, \om) ~:=~ \sup_{\P \in \Pc(t, \om)} \E^{\P} \big[ g(X)\big],\; (t,\om)\in [0,T]\x D([0,T]).
	$$
	We simply write $V(0, x_0)$ for $V(0,\cdot)$. 
	We also denote by $\Hc$ the collection of all $\R^d$--valued $\F^{\circ}$--predictable processes $H$ such that $\int_0^T H^{\top}_t d \langle X \rangle_t H_t < \infty$, $\P$--a.s. and $\int_0^{\cdot} H_s dX_s$ is a $\P$--supermartingale, for all $\P \in \Pc$.
 
 	\begin{Proposition} \label{prop:superreplication}
 		Let Assumption \ref{assum:DPP} hold true, and suppose in addition that $(V, X, \P)$ satisfies the conditions {of} Theorem \ref{thm: Ito} for each $\P \in \Pc$.
		Then
		 $V(\cdot, X)$ is a $\P$--supermartingale for every $\P \in \Pc$ and
		\begin{equation} \label{eq:duality}
			V(0, x_0) = \inf\Big\{ x ~: x + \int_0^T H_t dX_t \ge g(X),~H \in \Hc,~ \P \mbox{--a.s. for all}~\P \in \Pc \Big\}.
		\end{equation}
		Moreover, the superhedging problem at the r.h.s. of \eqref{eq:duality} is achieved by {$H^*:= \nabla_{\xr} V(\cdot, X)$}.
	\end{Proposition}
	\proof First, it is clear that one has the weak duality
	$$
		V(0, x_0) ~\le~ \inf\Big\{ x \in \R ~: x + \int_0^T H_t dX_t \ge g(X),~H \in \Hc,~ \P \mbox{--a.s. for all}~\P \in \Pc \Big\}.
	$$
	Next, {our} stability conditions  under conditioning and concatenation {of} Assumption \ref{assum:DPP} {imply the dynamic programming principle}.
	\begin{equation} \label{eq:DPP}
		V(t, \om) = \sup_{\P \in \Pc(t, \om)} \E^{\P} \big[ V(t+h, X) \big],
	\end{equation}
	{see e.g.~\cite{karoui2013capacities, el2013capacities}.}
	Together with the fact that $V \in \mathbb{C}^{0,1}$, this implies that $V(\cdot, X)$ is a  $\P$--continuous supermartingale for every $\P \in \Pc$.
	By Corollary \ref{coro:DoobMeyer}, one has
	$$
		V(0, x_0) + \int_0^T  {  \nabla_{\xr} V(t, X)}dX_t \ge g(X), ~\P \mbox{--a.s. for all}~\P \in \Pc.
	$$
	This implies   the duality result   \eqref{eq:duality} as well as {the fact that} $H^* := \nabla_{\xr} V(\cdot, X)$ is the optimal superhedging strategy.
	\qed
 
 	\begin{Remark}
		The duality result \eqref{eq:duality} in the model independent setting has been much investigated, see e.g.~ \cite{denis2006theoretical,neufeld2013superreplication}.
		In most cases, one {obtains} the existence of {an} optimal strategy $H^*$ but without an explicit expression.
		{I}n \cite{neufeld2013superreplication}, the duality is obtained for just measurable payoff functions $g$, but they require {$\Pc$ to be made of extremal} martingale measures.
		Our duality result {of} Proposition \ref{prop:superreplication} does not requires $\P \in \Pc$ to be extremal, but {requires} regularity conditions on the value function $V$.
		{This in turn allows us to characterize} the optimal superhedging strategy  explicitly as {the} Dupire's vertical derivative of the {pricing} function, 
		which also justify {the} initial motivation {of Dupire} \cite{dupireito} to introduce this notion of derivative.
	\end{Remark}

	\begin{Remark}
		The main idea in Propositions \ref{prop:replication} and \ref{prop:superreplication} is to show that the replication or super-replication prices of the  options are supermartingales, so that one can apply the Doob-Meyer decomposition result {of} Corollary \ref{coro:DoobMeyer}.
		We can also apply the same technique to other situations, such as the hedging of American options, the superhedging problems under constraints, etc., 
		{in which} the {option} price process {has} a natural supermartingale structure (see e.g.~\cite{bouchard2016fundamentals}).
	\end{Remark}

\subsection{Verification of the $\Cb^{0,1}(\Theta)$-regularity in a model with (bounded) uncertain volatility}

	Let us consider a more concrete superhedging problem in the context of an uncertain volatility model.
	Let $d=1$, $x_0 \in \R$, $0 \le \underline \sigma < \overline \sigma$ be fixed,
	we denote by  $\Pc_0$ the collection of all probability measures $\P$ such that $\P[X_0 = x_0] = 1$ and
	\begin{equation} \label{eq:vol_uncert}
		d X_s =  \sigma_s  dW^{\P}_s,~\sigma_s \in [\underline \sigma , \overline \sigma],~s \in [0,T], ~\P\mbox{-a.s.}
	\end{equation}
	for some $\P$--Brownian motion $W^{\P}$.
	We then consider a derivative option with payoff function $g:D([0,T]) \longrightarrow \R$ satisfying the following conditions.
	
	\begin{Assumption} \label{assum:Condition_g}
		$\mathrm{(i)}$ The function $g$ is bounded,
		and there exist  $\alpha \in (0,1]$ and  a finite positive measure $\mu$ on $[0,T]$ with at most finitely many atoms such that,
		for all $\xr,\xr'\in \DT$,   $B = [s, t) \subset [0,T]$ and $\delta\in \R$,
		\begin{equation}\label{eq: g Lispch}
			|g(\xr)-g(\xr')|\le  \int_{0}^{T}|\xr_{s}-\xr_{s}'|\mu(ds),
		\end{equation} 
		$\delta'\in \R\mapsto g(\xr+\delta'\1_{B})$   is differentiable and 
		\begin{equation}\label{eq: measure sur derive}
			\Big| \frac{d g(\xr+\delta\1_{B}+\xr')}{d\delta} - \frac{ d  g(\xr+\delta\1_{B})}{d \delta} \Big |
			\le 
			\Big( \int_{0}^{T} \big| \xr'_{s} \big| \mu(ds) \Big)^{\alpha} \mu(B).
		\end{equation}
 		
		\noindent $\mathrm{(ii)}$ for any increasing sequence $0=t_{0}<t_{1}<\cdots<t_{n}=T$ with $\max_{i<n}|t_{i+1}-t_{i}|$ small enough, 
		for all $1\le i<j< n$, there exists  $p^{i,j} \neq 0$ such that, 
		for all   $\delta \in \R$, and $(x_{\ell})_{0\le \ell \le n-1}\subset \R^{n}$,
		\begin{align} \label{eq: hyp pi}
			&g\left( \sum_{\ell=0}^{n-1} (x_{\ell}+ { \delta} \1_{\{\ell = i\}})\1_{[t_{\ell},t_{\ell+1})} + \1_{\{T\}} x_{{n-1}}\right) \nonumber \\
			&= g\left(\sum_{\ell=0}^{n-1} (x_{\ell}+p^{i,j} {\delta}\1_{\{\ell\ge j\}})\1_{[t_{\ell},t_{\ell+1})} + \1_{\{T\}} (x_{{n-1}}+p^{i,j} {\delta}) \right).
		\end{align}
	\end{Assumption}

	\begin{Remark}
		Let 
		$$
			g(X) ~=~ g_{\circ} \Big( \int_0^T X_t \mu_0(dt) \Big),
		$$
		where $g_{\circ} \in C^{1+\alpha}(\R)$ is bounded, and $\mu_0$ is a finite positive measure with at most  finitely many atoms on $[0,T]$ satisfying $\mu_0([T-h, T]) \neq 0$ for all $h > 0$ small enough.
		Then it satisfies Assumption \ref{assum:Condition_g}.
	\end{Remark}

	For each $(t, \xr) \in [0,T] \x D([0,T])$, we define
	\begin{equation} \label{eq:def_Pct}
		\Pc(t, \xr) :=  
		\big\{ \P  \in   \mathfrak{B}(\Om^{\circ})~: 
			\P [ X_{t \wedge } = \xr_{t \wedge } ] = 1,~\mbox{and \eqref{eq:vol_uncert}~holds on}~[t, T]
		\big\},
	\end{equation}
	and
	$$
		V(t,\xr) ~:= \sup_{\P \in \Pc(t, \xr) }\E^{\P} \big[ g(X) \big].
	$$
	\begin{Proposition}
		Let $\Pc_0$ and $g$ be given as above.
		Then, $V$ is  vertically differentiable  and the duality result \eqref{eq:duality} holds true with the optimal superhedging strategy $H^* := \nabla_{\xr} V(\cdot, X)$.
	\end{Proposition}
	\begin{proof}
	First, by rewriting $\P \in \Pc(t, \xr)$ as solution of a controlled martingale problem, it is easy to check that 
	the graph set $[[\Pc_0]]$ is a closed set, and satisfies the stability conditions under conditioning and concatenation (see e.g. \cite[Section 4]{el2013capacities}),
	so that Assumption \ref{assum:DPP} holds true.
	As in Proposition \ref{prop:superreplication}, one has the dynamic programming principle \eqref{eq:DPP}, 
	and consequently, $V(\cdot, X)$ is a $\P$--supermartingale for every $\P \in \Pc_0$ .
	
	\vspace{0.5em}

	Next, we assume that $\mu$ has possible atoms on $\{0 = T_0 < T_1 < \cdots < T_n = T\}$.
	Then by Propositions \ref{eq: v C 0+alpha 0+1} and \ref{prop: v C 0+alpha/2 1+alpha} below, 
	together with Propositions \ref{prop: cn pour cns} and \ref{prop: cond pour J2 sans carre},
	it follows that $V$ satisfies \eqref{eq:CNS_C01} and  the $\Cb^{0,1}$--regularity as well as other conditions  required in Theorem \ref{thm: Ito} on each interval $[T_k+\eps, T_{k+1}]$, for all $k=0, \cdots, n-1$, and all $\eps > 0$ small enough.
	Recalling that $V(\cdot, X)$ is a supermartingale under each $\P \in \Pc_0$, 
	it follows by Corollary \ref{coro:DoobMeyer} that
	$$
		V(T_{k+1}, X) - V(T_k+\eps, X) 
		~\le~ 
		\int_{T_k+\eps}^{T_{k+1}} H^*_t dX_t,
		~\mbox{with}~
		H^*_t := \nabla_{\xr} V(t, X).
	$$
	Taking the sum on $k = 0, \cdots, n-1$ and then letting $\eps \longrightarrow 0$, 
	we can then conclude as in Proposition \ref{prop:superreplication} to obtain the duality result \eqref{eq:duality} and that $H^*$ is the optimal strategy.
 	\end{proof}

	\begin{Remark}
		The regularity property of $V$ is given in Propositions \ref{eq: v C 0+alpha 0+1} and \ref{prop: v C 0+alpha/2 1+alpha}  below, which seems to be original in the literature.
		Moreover, it can be naturally extended to  payoff functions of the form $g \big(\int_0^T \xr_t \rho_1(dt), \cdots, \int_0^T \xr_t \rho_m(dt) \big)$, for finitely many   measures $\rho_1, \cdots, \rho_m$.
		We nevertheless restrict to the one measure case to make the presentation more accessible.
	\end{Remark}

	\begin{Proposition}\label{eq: v C 0+alpha 0+1} 
		Let Assumption \ref{assum:Condition_g} hold true.
		Then for all $(t,\xr,\xr',h)\in [0,T]\x D([0,T])\x D([0,T])\x\R_{+}$ with $t+h\le T$, one has
		\begin{equation} \label{eq:V_continuity}
			|V(t+h,\xr_{t \wedge })- V(t,\xr)|\le \overline \sigma h^{\frac12} \mu([t,T])
			~\mbox{and}~
			|V(t,\xr')- V(t,\xr)|\le   \int_{0}^{t}|\xr'_{s}-\xr_{s}|\mu(ds).
		\end{equation}
	\end{Proposition}  
	\begin{proof} It suffices to observe that \eqref{eq: g Lispch} implies that 
	\begin{align*}
		\big| V(t+h,\xr_{t \wedge })- V(t,\xr) \big|
		&\le
		\sup_{\P \in \Pc(t, \xr)} \E^{\P} \Big[
			\int_{t}^{T} \Big| \xr_{t}+\int_{t+h}^{s\vee (t+h)}\sigma_{r}dW^{\P}_{r}-\xr_{t}-\int_{t}^{s}\sigma_{r}dW^{\P}_{r} \Big| \mu(ds)
			\Big]\\
		&\le 
		\sup_{\P \in \Pc(t, \xr)} \E^{\P} \Big[ \int_{t}^{T} \Big| \int_{t}^{s\wedge (t+h)}\sigma_{r}dW^{\P}_{r}  \Big|\mu(ds) \Big].
	\end{align*}
	The second estimate is also an immediate consequence of \eqref{eq: g Lispch}. 
	\end{proof}

	\begin{Proposition}\label{prop: v C 0+alpha/2 1+alpha} 
		Let Assumption \ref{assum:Condition_g} hold true.
		Then the vertical derivative $\nabla_{\xr} V(t,\xr)$ is well defined for all $(t, \xr) \in [0,T]\x\DT$, and there exists $C>0$ such that 
		$|\nabla_{\xr} V(t,\xr)|\le C$,
		\begin{equation} \label{eq:DV_vertical_continuity}
			|\nabla_{\xr} V(t,\xr')- \nabla_{\xr} V(t,\xr)|
			~\le~
			C \Big( \Big|\int_{0}^{t} |\xr'_{s}-\xr_{s}| \mu(ds) \Big|^{\alpha}+|\xr'_{t}-\xr_{t}|^{\alpha}\Big),
		\end{equation}
		and 
		\begin{align}\label{eq: unif conti en temps de nabla v}
			|\nabla_{\xr} V(t',\xr_{t \wedge })- \nabla_{\xr} V(t,\xr)|
			~\le~
			C \big( |t'-t|^{\frac{\alpha}{2+2\alpha}}+\mu([t,t')) \big),
		\end{align}
		for all $t\le t'\le T$ and $\xr,\xr'\in \DT$. 
	\end{Proposition}
	\begin{proof}
	Without loss of generality, we restrict to the collection $D_0([0,T])$ of c\`adl\`ag paths $\xr$ with initial condition $\xr_0 = x_{0}$, where $x_0 \in \R$ is the constant introduced above \eqref{eq:vol_uncert}.

	\vspace{0.5em}
	
	{\rm 1.}  Let us consider a sequence $(\pi^n)_{n \ge 1}$ of discrete time grids,  dense in $[0,T]$, such that $\pi^{n}=(t^{n}_{i})_{0\le i\le n}\subset [0,T]$ and $\{0,T\}\subset \pi^n \subset \pi^{n+1}$ for all $n \ge 1$, and $\max_{0 \le i \le n-1} |t^n_{i+1} - t^n_i| \longrightarrow 0$ as $n \longrightarrow \infty$.
	Remembering that $\mu$ has at most finitely many atoms on $[0,T]$,
	one can choose $(\pi^n)_{n \ge 1}$ such that $\{t\in  [0,T] : \mu(\{t\})> 0\} \subset \cup_{n \ge 1} \pi^n$.
	Next, let us define, for all  $n \ge 1$ and  $(t,\xr,x)\in [0,T]\x D_0([0,T])\x \R $,
	$$
		V^{n}(t, \xr, x):=\!\!\! \sup_{\P \in \Pc(t, \xr \oplus_{t}(x- \xr_{t}))} \!\!\! \E^{\P} \big[ g \big(\Pi^{n}[\xr\1_{[0,t^{n}_{i+1})}+X\1_{[t^{n}_{i+1},T]}] \big) \big]\;
		\mbox{ if } t\in [t^{n}_{i},t^{n}_{i+1}),\; i\le n-1,
	$$
	where
	$$
		\Pi^{n}[\xr]:=\sum_{i=0}^{n-1} \xr_{t^{n}_{i}} \1_{[t^{n}_{i},t^{n}_{i+1})}+\xr_{t^{n}_{n-1}} \1_{\{T\}}.
	$$
	Notice that, for $t \in [t^n_i, t^n_{i+1})$, $V^n(t, \xr, x)$ depends on $(\xr, x)$ only through $(\xr_{t^n_1}, \cdots, \xr_{t^n_i}, x)$.
	This motivates us to introduce $\Pi^{n,i}_t : \R^{i+1} \to D_0([0,T])$ defined   for $i<n$  by
	$$
		\Pi^{n,i}_t (y_1, \cdots, y_i, x) ~:=~ \sum_{j=0}^{i-1} y_j \1_{[t^n_j, t^n_{j+1})} + y_i \1_{[t^n_i, t)} + x \1_{[t, T]}, \; t\in (t^{n}_{i},t^{n}_{i+1}]
	$$
	as well as
	$$
		g^{n}(y_{1},\ldots,y_{n-2},x)
		~:=~
		g \Big( \Pi^{n, n-2}_{t^n_{n-1}} \big(x_{0},y_1, \cdots, y_{n-2}, x \big) \Big),
	$$
	and
	$$
		v^{n}(t,y_{1},\ldots,y_{i},x) 
		~:=~
		V^n \Big(t, \Pi^{n,i}_{t^n_{i+1}} \big(y_1, \cdots, y_i \big), x \Big),\;t\in [t^{n}_{i},t^{n}_{i+1}).
	$$
	Notice that, for all $t \in [t^n_i, t^n_{i+1})$, $\xr \in D_0([0,T])$, $x \in \R$,
	$$
		V^n(t, \xr, x) ~=~ V^n(t, \bar \xr^n, x) ~=~ v^n(t, \xr_{t^n_1}, \cdots, \xr_{t^n_i}, x),
		~~\mbox{with}~
		\bar \xr^{n}:=\Pi^{n}[\xr].
	$$
	We further observe from \eqref{eq: g Lispch} that, for all $(t,\xr)\in [0,T]\x C([0,T])$ with $t\in [t^{n}_{i_{\circ}-1},t^{n}_{i_{\circ}})$ for some $1<i_{{\circ}}\le n$, 
	\begin{align*}
		\big| V^{n}(t,\bar \xr^{n},\xr_{t})- V(t,\xr) \big|
		\le&  \int_{0}^{t^{n}_{i_{\circ}}}|\bar \xr^{n}_{s}-  \xr_{s}| \mu(ds)
			+ \sum_{i=i_{\circ}}^{n-1} \int_{t^{n}_{i}}^{t^{n}_{i+1}}\sup_{\P \in \Pc_{0}} \E^{\P} \Big[ \Big |\int_{t^{n}_{i}}^{s} \sigma_{r} dW^{\P}_{r} \Big| \Big]\mu(ds) 
		\\
		\le &   \int_{0}^{t^{n}_{i_{\circ}}}|\bar \xr^{n}_{s}-  \xr_{s}| \mu(ds)
			+ \sum_{i=i_{\circ}}^{n-1} \int_{t^{n}_{i}}^{t^{n}_{i+1}} \overline \sigma (s-t^{n}_{i})^{\frac12}\mu(ds) 
		\\
		\le &   \int_{0}^{t^{n}_{i_{\circ}}}|\bar \xr^{n}_{s}-  \xr_{s}| \mu(ds)
			+  \overline \sigma \Big( \max_{0 \le i \le n-1} |t^n_{i+1} - t^n_i|^{\frac12} \Big) \mu([0,T]).
	\end{align*}
	 As $\mu$ is a finite   measure on $[0,T]$, it follows that
	\begin{equation} \label{eq: conv uniforme de vn}
		\big| V^{n}(t, \xr,  \xr_{t}) - V(t,\xr) \big| 
		= 
		\big| V^{n}(t,\bar \xr^{n},  \xr_{t}) - V(t,\xr) \big|
		\longrightarrow 0,
		~~\mbox{as}~ n\longrightarrow \infty.
	\end{equation}

	{\rm 2.}  Let us set 
	$$
		F: \gamma \in \R
		~\longmapsto~
		\max_{a \in [\underline \sigma^2, \overline \sigma^2]} \frac12 a \gamma
		~=~
		 \frac12 \overline \sigma^{2} \gamma^{+}-\frac12\underline \sigma^{2} \gamma^{-}.
	$$ 
	Then for each $i\le n-2$, $v^{n}$ is a continuous viscosity solution of 
	\begin{align}
		&\partial_{t}  v^{n}(t, z)+F(D^{2}  v^{n}(t, z))=0\; \mbox{ for } (t,z)\in [t^{n}_{i},t^{n}_{i+1})\x \R^{i+1}, 
		\label{eq: EDP int dom}\\
		&\lim_{t\uparrow t^{n}_{i+1},(y',x')\longrightarrow (y,x)}  v^{n}(t,y',x')= v^{n}(t^{n}_{i+1},y,x,x),\mbox{ for } (y,x)\in \R^{i}\x \R,  
		\label{eq: EDP cond bord ti}
	\end{align}
	with terminal condition
	\begin{equation} \label{eq: EDP cond bord en T}
		 v^{n}(t^{n}_{n-1},\cdot)= g^{n}.
	\end{equation}
	In the above,  $D v^{n}$ and $D^{2} v^{n}$ denote for the first and second order derivative with respect to the last argument of $ v^{n}$. 
	The operator $F$ being Lipschitz, it follows from standard arguments  that this system satisfies a comparison principle among (semi-continuous) bounded viscosity solutions. 

	\vspace{0.5em}

	Let us denote by $D_i g^n$ the partial derivative of $g^n$ w.r.t. the $i$-th argument, then by \eqref{eq: g Lispch}-\eqref{eq: measure sur derive}, for all $z, z' \in \R^{n-1}$,
	$$
		\big| D_i g^n (z) \big| \le \mu \big( [t^n_i, t^n_{i+1}) \big),
		~~
		\big| D_i g^n( z+ z') - D_i g^n(z) \big| \le
		\Big(\!\! \int_0^T \!\! \big| \Pi^{n, {n-2}}_{t^n_{n-1}}[z']_t \big|  \mu(dt) \Big)^{\alpha} \mu \big([t^n_i, t^n_{i+1}) \big) .
	$$
	We will next regularize $(g^n, F)$. 
	Let $\rho_k : \R^{n-1} \to \R_+$ be a $C^{\infty}$ density function with compact support,
	and $g^n_k := g^n * \rho_k$ be the regularized function obtained by convolution.
	Then it is clear that $g^n_k$ still satisfies
	\begin{equation} \label{eq:estimation_gnk}
		\big| D_i g^n_k (z) \big| \le \mu \big( [t^n_i, t^n_{i+1}) \big),
		~~
		\big| D_i g^n_k( z + z') - D_i g^n_k(z) \big| 
		\le
		\Big(\!\! \int_0^T \!\! \big| \Pi^{n, {n-2}}_{t^n_{n-1}}[z']_t \big|  \mu(dt) \Big)^{\alpha} \mu \big([t^n_i, t^n_{i+1}) \big).
	\end{equation}
	Moreover, it follows from \eqref{eq: hyp pi} that for each $i < j$, there exists $p^{i,j} \neq 0$ such that, for all $z \in \R^{n-1}$, $\delta \in \R$, 
	\begin{equation} \label{eq:gnk_pij}
		g^n_k \big( z + \delta e^{n-1}_i \big) 
		~=~
		g^n_k \Big( z + p^{i,j} \delta \sum_{\ell=j}^{n-1} e^{n-1}_j \Big),
	\end{equation}
	where $e^{n-1}_i  ~\mbox{denotes the}~i \mbox{-th standard unit vector in}~\R^{n-1}$.
	
	Since $F$ is a convex function, one can approximate it by a $C^{\infty}$ convex function $F_k$, $k\ge 1$, such that
	$$
		F_k(\gamma) = 
		\begin{cases}
			\frac12 \overline \sigma^2 \gamma , &\mbox{for}~\gamma \ge 1, \\
			\frac12 \underline \sigma^2_k \gamma, &\mbox{for}~\gamma \le -1,
		\end{cases}
		~~\mbox{with}~ \underline \sigma^2_k := \underline \sigma^2 \vee k^{-1}.
	$$
	Let $F^*_k(a)=\sup_{\gamma \in \R} (a \gamma - F_k( \gamma))$ be the Fenchel transformation of $F_k$,
	so that 
	$$
		F_k( \gamma) 
		~=
		\sup_{a \in  [\frac12 \underline \sigma^2_k, \frac12\bar \sigma^2]} \big(a \gamma - F^{*}_k(a) \big).
	$$	
	Let $v^n_k$ be the corresponding solutions of   \eqref{eq: EDP int dom}-\eqref{eq: EDP cond bord ti}-\eqref{eq: EDP cond bord en T} with parameters $(F_k, g^n_k)$ such that $(F_k, g^n_k)\to (F,g^{n})$ as $k\to \infty$.
	Then $v^n_k \in C^{1,3}_b$ and 
	\begin{equation} \label{eq:vnk2vn}
		v^n_k \longrightarrow v^n,~\mbox{pointwise as $k\to \infty$}.
	\end{equation}
	Moreover, as $F_k$ is convex, the associated equations on $v^n_k$ is still a HJB equation,
	so that $v^n_k$ can be considered as the value function of a control problem:
	\begin{equation} \label{eq:CtrlPb_vnk}
		v^n_k(t, y,x) 
		=
		\sup_{\alpha \in \Ac_k} 
		\E \Big[ g^n_k \big(y_{1}, \cdots, y_i, X^{t,x, \alpha}_{t^n_{i+1}}, \cdots, X^{t,x, \alpha}_{t^n_{n-1}} \big)
			- \int_t^T F^*_k(\frac12 \alpha_s^2) ds
		\Big],
		~~t \in [t^n_i, t^{n}_{i+1}),
	\end{equation}
	where $X^{t,x, \alpha}_s := x + \int_t^s \alpha_r dW_r,~s \ge t$, and $\Ac_k$ is the collection of all progressively measurable process $\alpha$ taking value in $[\underline \sigma_k, \overline \sigma]$ on some filtered probability space equipped with a Brownian motion $W$.
	
	\vspace{0.5em}

	{\rm 3.} For $i \le n -2$, $t \in [t^{n}_{i},t^{n}_{i+1})$, and $(\xr, x)\in D_0([0,T]) \x \R$, let us set
	$$
		V^{n}_{k}(t, \xr, x) 
		~:=~
		v^{n}_{k}(t,\xr_{t^{n}_{1}},\ldots,\xr_{t^{n}_{i}},x).
	$$
	We claim that there exists a constant $C>0$ such that, 
	for all $k,n\ge 1$, $t\in [t^{n}_{i},t^{n}_{i+1})$, $h \in (0, T-t]$, $\xr,\xr'\in D_0([0,T])$,
	\begin{eqnarray}
		&
		|DV^n_k(t, \xr, x)| 
		~\le~ 
		C,
		&
		\label{eq:bounded_Vnk}\\
		&
		\big| D  V^{n}_{k}(t, \xr', \xr'_{t})-D V^{n}_{k}(t, \xr, \xr_{t}) \big|
		~\le~
		C \Big( \big| \int_{0}^{t} |\bar \xr^{'n}_{s}-\bar \xr^{n}_{s}| \mu(ds) \big|^{\alpha} + \big| \xr_t' - \xr_t \big|^{\alpha} \Big),
		&
		\label{eq: borne C1+alpha si smooth nk bis}\\
		&
		\big| D   V^{n}_{k}(t+h, \xr_{t \wedge \cdot},\xr_{t})-D   V^{n}_{k}(t, \xr, \xr_{t}) \big|
		~\le~
		C \big( h^{\frac{\alpha}{2+2\alpha}}+\mu([t,t+h)) \big),
		&
		\label{eq: borne holder temps nk bis}
	\end{eqnarray}
	where $D V^{n}_{k}$ denote the derivative of $V^{n}_{k}$ with respect to its last argument. 
	Then, for each $t\le t'\in [0,T]$, $\xr,\xr' \in D_0([0,T])$ and $x,x'\in \R$, 
	\begin{equation} \label{eq:DVnk_taylor}
		\big| V^{n}_k(t, \xr, x)-V^{n}_k(t, \xr, x')- D V^{n}_k (t, \xr, x)(x-x') \big|
		~\le~ 
		C \big| x-x' \big|^{1+\alpha}. 
	\end{equation}

	Next, let $\T := \cup_{n \ge 1} \pi^n$ and $\Q$ be the set of all rational numbers,
	so that $\T \x \Q$ is a countable dense subset of $[0,T] \x \R$.
	We then define a countable subset $Q_T$ of $[0,T] \x D_0([0,T]) \x \R$ by
	$$
		Q_T
		~:=~
		\cup_{n \ge 1}  
		\big\{
			(t, \xr, x) ~: t \in \T, ~ \xr = \Pi^n[\xr], ~x \in \Q, ~ \xr_s \in \Q \mbox{ for } s \in [0,T]
		\big\}.
	$$
	In view of \eqref{eq:bounded_Vnk} and  the convergence results \eqref{eq: conv uniforme de vn} and \eqref{eq:vnk2vn},
	one can extract a subsequence $(n_{\ell}, k_{\ell})_{\ell \ge 1}$, such that, for all $(t, \xr, x) \in Q_T$,
	$$
		\big( V^{n_{\ell}}_{k_{\ell}} (t, \xr, \xr_t), D V^{n_{\ell}}_{k_{\ell}} (t, \xr, x) \big)
		~\longrightarrow~
		\big( V(t, \xr), DV(t, \xr, x) \big),
		~\mbox{as}~\ell \longrightarrow \infty,		
	$$
	for some function $DV: Q_T \longrightarrow \R$.
	Moreover, by \eqref{eq:bounded_Vnk}-\eqref{eq: borne C1+alpha si smooth nk bis}-\eqref{eq: borne holder temps nk bis}
	and \eqref{eq:DVnk_taylor}, $DV$ satisfies
	$$
		\big| V(t,  \xr \oplus_t x ) - V(t,  \xr \oplus_t x') - DV(t,  \xr , x ) (x - x' ) \big| 
		~\le~ 
		C |x - x' |^{1+\alpha},
	$$
	and 
	\begin{eqnarray}
		&&
		\big| DV(t,  \xr, x )- DV(t',  \xr', x') \big| \nonumber \\
		&\le& 
		C\Big( \Big| \int_{[0,t')} \!\!\!\!  \big| \xr_{t\wedge s} - \xr'_{s} \big| \mu(ds) \Big|^{\alpha} 
			\!\!+ 
			\big| x- x' \big|^{\alpha}  
			\!\!+
			|t-t'|^{\frac{\alpha}{2+2\alpha}}+\mu([t-t'))
		\Big),
		\label{eq:continuity_DV}
	\end{eqnarray}
	for all  $(t,\xr,x), (t', \xr', x')\in Q_T$ such that $t \le t'$.
	Notice that under the distance
	$$
		\rho\big( (t,\xr, x), (t', \xr', x') \big)
		:=
		 \int_{[0,t')} \!\!\!\!  \big| \xr_{t\wedge s} - \xr'_{s} \big| \mu(ds) \Big| 
			\!\!+ 
			\big| x- x' \big|
			\!\!+
			|t-t'| + \mu([t-t')),
	$$
	$Q_T$ is a dense subset of $[0,T] \x D_0([0,T]) \x \R$.
	Then, by continuity  of $V$ and $DV$, recall \eqref{eq:V_continuity} and \eqref{eq:continuity_DV}, 
	one can extend the definition of $DV$ to $[0,T]\x D_0([0,T]) \x \R$
	in such a way  that  $DV(t, \xr, \xr_t) = \nabla_{\xr} V(t, \xr)$ for all $(t,\xr) \in [0,T] \x D_0([0,T])$,
	and $\nabla_{\xr} V$ is uniformly bounded and satisfies \eqref{eq:DV_vertical_continuity}-\eqref{eq: unif conti en temps de nabla v}.

 	\vspace{0.5em}
 
	{\rm 4.} It remains to prove \eqref{eq:bounded_Vnk}-\eqref{eq: borne C1+alpha si smooth nk bis}-\eqref{eq: borne holder temps nk bis}. 

	\vspace{0.5em}

	{\rm a.} We start by proving \eqref{eq:bounded_Vnk}.
	Recall that $v^{n}_{k} \in C^{1,3}_b$.
	Let us denote by $\phi^{n,j}_{k}$ the derivative of  $v^{n}_{k}$ in its $j$-th space argument. 
	For all $i\le n-2$ and $j\le i+1$, it solves 
	\begin{equation}
		\partial_{t} \phi^{n,j}_{k}(t,z)+ F_{k}'(D^{2}  v^{n}_{k}(t,z))D^{2}\phi^{n,j}_{k}(t,z)=0,
		~~ (t,z)\in [t^{n}_{i},t^{n}_{i+1})\x \R^{i+1}, \label{eq: EDP Dv int dom}
	\end{equation}
	with the boundary condition
	\begin{equation}
		\lim_{t'\uparrow t^{n}_{i+1}}\phi^{n,j}_{k}(t,y,x)=\phi^{n,j}_{k}(t^{n}_{i+1},y,x,x)+\1_{\{j=i+1\}}\phi^{n,i+2}_{k}(t^{n}_{i+1},y,x,x),
		~~ (y,x)\in \R^{i}\x \R,  \label{eq: EDP Dv cond bord ti}
	\end{equation}
	where 
	\begin{equation}
		\phi^{n,j}_{k}(t^{n}_{n-1},\cdot)=D_{j} g^{n}_{k},\;\mbox{ for } j\le n-1. 
		\label{eq: EDP Dv cond bord T} 
	\end{equation}
	As $x\in \R\mapsto F_{k}'(D^{2}  v^{n}_{k}(t,y,x))$ is Lipschitz,
	it follows from the Feynman-Kac formula that, for all  $t\in [t^{n}_{i},t^{n}_{i+1})$, $x\in \R^{i+1}$ and $j\le i+1$,
	$$
		\phi^{n,j}_{k}(t,z)
		~=~
		\E \Big[ \Big( D_{j} g^{n}_{k}+\1_{\{j=i+1\}}\sum_{j'=i+2}^{n-1} D_{j'} g^{n}_{k} \Big) \big( \Pi^n( Y^{t,z} ) \big) \Big],
	$$
	for some process $Y^{t,z}$. 
	Then the first inequality in \eqref{eq:estimation_gnk} implies that $| \phi^{n,j}_{k}|\le \mu([0,T])$.

	\vspace{0.5em}

	{\rm b.} We now prove \eqref{eq: borne C1+alpha si smooth nk bis}. 
	Let us fix  $i\le n-2$, $j, \ell \le i+1$, $z' \in \R^{\ell}$ and then define, 
	for $t\in [t^{n}_{i},t^{n}_{i+1})$, $z\in \R^{i+1}$,
	$$
		\psi^{n,j, [z']_{\ell}^{i+1}}_{k}(t,z)
		~:=~
		\Big( \phi^{n,j}_{k} \big( t,z+ [z']_{\ell}^{i+1} \big) - \phi^{n,j}_{k}(t,z) \Big),
	$$
	where $[z']_{\ell}^{i+1} = (z'_1, \cdots, z'_{\ell}, 0, \cdots 0) \in \R^{i+1}$.
	Using \eqref{eq: EDP Dv int dom}, one obtains that,
	\begin{align*}
		0 ~=~ &\partial_{t} \psi^{n,j, [z']_{\ell}^{i+1}}_{k}(t,z)
			+ F_{k}'(D^{2}  v^{n}_{k}(t,z+ [z']_{\ell}^{i+1}))D^{2}\psi^{n,j,[z']_{\ell}^{i+1}}_{k}(t, z) \\
		&\;\;\;
			+~ F_{k}'' \big( A^{[z']_{\ell}^{i+1}} (t,z) \big) D^{2} \phi^{n,j}_{k}(t,z) D \psi^{n, {i+1}, [z']_{\ell}^{i+1}}_{k}(t,z),
		~~(t,z) \in [t^{n}_{i},t^{n}_{i+1}) \x \R^{i+1},
	\end{align*}
	in which $A^{[z']_{\ell}^{i+1}}$ is a continuous function. 
	We now observe that \eqref{eq:gnk_pij}-\eqref{eq:CtrlPb_vnk} imply that
	$$
		\phi^{n,j}_{k}(t,\cdot)=\phi^{n,i+1}_{k}(t,\cdot) p^{j,i+1} 
	$$
	for some $p^{j,i+1}\ne 0$. 
	Hence, $\psi^{n,j, [z']_{\ell}^{i+1}}_{k}$ satisfies the PDE
	\begin{align*}
		0~=~&\partial_{t} \psi^{n,j, [z']_{\ell}^{i+1}}_{k}(t,z)
			+~F_{k}' \big( D^{2}  v^{n}_{k}(t,z+[z']_{\ell}^{i+1}) \big) D^{2}\psi^{n,j,[z']_{\ell}^{i+1}}_{k}(t,z)\\
		&\;\;\;\;
			+~ \frac1{p^{j,i+1} }\Big[F_{k}''(A^{[z']_{\ell}^{i+1}}(t,z))D^{2} \phi^{n,j}_{k}(t,z)\Big] D \psi^{n, j,[z']_{\ell}^{i+1}}_{k}(t,z),
		~~(t,z) \in [t^{n}_{i},t^{n}_{i+1}) \x \R^{i+1},
	\end{align*}
	and, by \eqref{eq: EDP Dv cond bord ti}-\eqref{eq: EDP Dv cond bord T},
	\begin{align*}
		\lim_{t'\uparrow t^{n}_{i+1}} \psi^{n,j,[z']_{\ell}^{i+1}}_{k}(t,y,x)
		=&
		\1_{\{\ell < i+1\}} \Big( \psi^{n,j, [z']_{\ell}^{i+2}}_{k} +\1_{\{j=i+1\}}\psi^{n,i+2, [z']_{\ell}^{i+2}}_{k} \Big) (t^{n}_{i+1},y,x,x)\\
		&
		+\1_{\{\ell=i+1\}} 
		\Big( \psi^{n,j, [[z']]_{\ell+1} }_{k} +\1_{\{j=i+1\}}\psi^{n,i+2,  [[z']]_{\ell+1} }_{k} \Big) (t^{n}_{i+1},y,x,x),
	\end{align*}
	for $(y,x)\in \R^{i}\x \R$, and
	$$
		\psi^{n,j, z''}_{k}(t^{n}_{n-1},\cdot)
		~=~
		\Delta_{z''} D_{j} g^{n}_{k}
		~:=~
		\big(D_{j} g^{n}_{k}(\cdot+z'' )-D_{j} g^{n}_{k} \big),
		~~\mbox{for all}~z'' \in \R^{n-1},
	$$
	in which
	$$
		[[z']]_{\ell+p} ~:=~ (z'_1, \cdots, z'_{\ell}, z'_{\ell}, \cdots, z'_{\ell}) \in \R^{\ell + p},\;p\ge 1.
	$$
	Then one can apply the Feynman-Kac formula to find a process $\tilde Y^{t,z}$  such that
	\begin{align*}
		\psi^{n,j,[z']_{\ell}^{i+1}}_{k}(t,z)
		=&
		\1_{\{\ell < i+1\}} 
		\E \Big[ \Big(\Delta_{[z']_{\ell}^{n-1}} D_{j} g^{n}_{k}
			+
			\1_{\{j=i+1\}} \!\!\! \sum_{j'=i+2}^{n-1} \Delta_{[z']_{\ell}^{n-1}}D_{j'} g^{n}_{k} \Big) 
			\big( \Pi^{n}[\tilde Y^{t,z}] \big) 
		\Big]\\
		&+\1_{\{\ell=i+1\}}  
		\E \Big[ \Big(\Delta_{[[z']]_{\ell+n_{\ell}}} D_{j} g^{n}_{k}
			+
			\1_{\{j=i+1\}} \!\!\! \sum_{j'=i+2}^{n-1} \Delta_{[[z']]_{\ell+n_{\ell}}}D_{j'} g^{n}_{k} \Big) 
			\big( \Pi^{n}[\tilde Y^{t,z}] \big) 
		\Big],
	 \end{align*}
	 with $n_{\ell}:=n-1-\ell$. 
	In view of the second inequality in \eqref{eq:estimation_gnk}, this concludes the proof of \eqref{eq: borne C1+alpha si smooth nk bis}.

	\vspace{0.5em}

	{\rm c}. We finally prove \eqref{eq: borne holder temps nk bis}. 
	In view of the representation of $v^n_k$ as the value function of an optimal control problem in \eqref{eq:CtrlPb_vnk},
	one can apply exactly the same arguments as  in Proposition \ref{eq: v C 0+alpha 0+1},
	together with \eqref{eq:estimation_gnk}, to obtain that,
	for all $t\le t'\in [0,T]$ and $\xr,\xr'\in D_0([0,T])$,
	$$
		\big| V^{n}_{k}(t', \xr_{ t\wedge}, \xr_{t})-  V^{n}_{k}(t, \xr, \xr_{t}) \big|
		~\le~
		\overline \sigma |t'-t|^{\frac12}\mu([t,T]),
	$$
	and
	$$
		\big| V^{n}_{k}(t, \xr', \xr_{t})-  V^{n}_{k}(t, \xr, \xr_{t}) \big|
		~\le~
		\int_{0}^{t} \big|\bar \xr^{'n}_{s}-\bar \xr^{n}_{s} \big| \mu(ds).
	$$
	Let us set $f := D  V^{n}_{k}$. It follows from the above estimations, together with \eqref{eq: borne C1+alpha si smooth nk bis}, that,
	\begin{align*}
		& \big| f(t+h^{2+2\alpha},   \xr_{t \wedge },\xr_{t} ) - f(t,  \xr, \xr_{t}) \big| \\
		&\le h^{-1} \big|V^{n}_{k}(t+h^{2+2\alpha}, \xr_{t \wedge },\xr_{t} +h) - V^{n}_{k}(t+h^{2+2\alpha}, \xr_{t \wedge },\xr_{t} ) - V^{n}_{k}(t, \xr, \xr_{t}+h)+ V^{n}_{k}(t, \xr,\xr_{t} ) \big|\\
		&\;\;+2Ch^{\alpha}\\
		&\le  h^{-1} \big| V^{n}_{k}(t+h^{2+2\alpha},( \xr  \oplus_{t} h)_{t \wedge },  \xr_{t}+h)- V^{n}_{k}(t, \xr ,  \xr_{t}+h) \big|\\
		&\;\; +  h^{-1} \big| V^{n}_{k}(t, \xr,\xr_{t} )- V^{n}_{k}(t+h^{2+2\alpha}, \xr_{t \wedge },\xr_{t} ) \big|\\
 		&\;\;+ h^{-1} \big| V^{n}_{k}(t+h^{2+2\alpha}, \xr_{t \wedge },\xr_{t}+h)- V^{n}_{k}(t+h^{2+2\alpha},( \xr  \oplus_{t} h)_{t \wedge },\xr_{t}+h) \big|\\
		&\;\;+2Ch^{\alpha}\\
		&\le 2Ch^{\alpha}+2\overline \sigma h^{\alpha}\mu([0,T])+\mu([t,t+h^{2+2\alpha})),
	\end{align*} 
	for all $\xr \in C_{x_{0}}([0,T])$ and $t< t+h\le T$.
	This concludes the proof of \eqref{eq: borne holder temps nk bis}.
\end{proof}

\begin{Remark}
Note that the same proof would go through if $F$ was affine  in place of assuming \eqref{eq: hyp pi}. In this case, the term $F''_{k}$ in Step 4.b.~of the proof of Proposition \ref{prop: v C 0+alpha/2 1+alpha} would simply be zero. This corresponds to the case where $\underline \sigma=\overline \sigma$. 
\end{Remark}

\bibliographystyle{plain}
\def\cprime{$'$} \def\cprime{$'$}

\end{document}